\documentclass[12pt,a4paper]{article}
\usepackage{amssymb}
\usepackage{amsmath}
\usepackage{latexsym}
\usepackage{amsthm}

\def\E{\mathbb E}

\newtheorem{lem}{Lemma}
\newtheorem{thm}{Theorem}

\newtheorem{prop}[thm]{Proposition}
\theoremstyle{definition}
\newtheorem*{remark}{Remark}
\theoremstyle{definition}

\usepackage{xpatch}
\xpatchcmd{\proof}{\itshape}{\normalfont\proofnameformat}{}{}
\newcommand{\proofnameformat}{}

\begin{document}

\renewcommand{\proofnameformat}{\bfseries}

\begin{center}
{\large\textbf{Random walks on the circle and Diophantine approximation}}

\vspace{5mm}

\textbf{Istv\'an Berkes$^1$ and Bence Borda$^{1,2}$}

\vspace{5mm}

{\footnotesize $^1$Alfr\'ed R\'enyi Institute of Mathematics

Re\'altanoda utca 13--15, 1053 Budapest, Hungary

\vspace{5mm}

$^2$Graz University of Technology

Steyrergasse 30, 8010 Graz, Austria

Email: \texttt{berkes.istvan@renyi.hu} and \texttt{borda@math.tugraz.at}}

\vspace{5mm}


\end{center}

\vspace{5mm}

\begin{abstract}
Random walks on the circle group $\mathbb{R}/\mathbb{Z}$ whose elementary steps are lattice variables with span $\alpha \not\in \mathbb{Q}$ or $p/q \in \mathbb{Q}$ taken mod $\mathbb{Z}$ exhibit delicate behavior. In the rational case we have a random walk on the finite cyclic subgroup $\mathbb{Z}_q$, and the central limit theorem and the law of the iterated logarithm follow from classical results on finite state space Markov chains. In this paper we extend these results to random walks with irrational span $\alpha$, and explicitly describe the transition of these Markov chains from finite to general state space as $p/q \to \alpha$ along the sequence of best rational approximations. We also consider the rate of weak convergence to the stationary distribution in the Kolmogorov metric, and in the rational case observe a surprising transition from polynomial to exponential decay after $\approx q^2$ steps; this seems to be a new phenomenon in the theory of random walks on compact groups. In contrast, the rate of weak convergence to the stationary distribution in the total variation metric is purely exponential.
\end{abstract}

\section{Introduction}

Convergence of random walks on compact groups to the uniform distribution (Haar measure) is a classical topic in probability theory. Such a convergence takes place under very general assumptions, but its finer properties are rather intricate. A classical example is the cutoff phenomenon of Aldous and Diaconis \cite{A}, \cite{AD}, \cite[Chapter 4]{DI} for random walks on the finite symmetric group corresponding to card shuffling, when the distance from uniformity in the total variation metric remains close to $1$ for a long time, then drops almost immediately near $0$. In this paper we consider random walks on the circle group $\mathbb{R}/\mathbb{Z}$, which also exhibit surprisingly delicate phenomena.

Let $X_1, X_2, \dots$ be independent, identically distributed (i.i.d.) nondegenerate integer-valued random variables, and put $S_k=\sum_{j=1}^k X_j$. Given an irrational $\alpha$, the sequence $S_k \alpha \pmod{\mathbb{Z}}$ is a random walk on $\mathbb{R}/\mathbb{Z}$ whose asymptotic behavior depends sensitively on the Diophantine approximation properties of $\alpha$. In the terminology of probability theory, $S_k \alpha \pmod{\mathbb{Z}}$ is a discrete time Markov chain on a general state space, with the uniform distribution on $\mathbb{R}/\mathbb{Z}$ as its stationary distribution. Note that the chain is irreversible, and, as we will see, its terms are polynomially weakly dependent. We can also consider a discrete analogue by choosing a reduced fraction $p/q$ instead of an irrational $\alpha$. In this case $S_k p/q \pmod{\mathbb{Z}}$ is a random walk (in particular, a discrete time Markov chain) on the finite cyclic group $\mathbb{Z}_q := \{ 0,1/q, \dots (q-1)/q \}$. The main purpose of this paper is to study the asymptotic behavior of these random walks, and to explicitly describe the transition from finite to general state space as $p/q \to \alpha$ along the sequence of best rational approximations to a given irrational $\alpha$.

Starting with the irrational case, let
\[ \psi (k) := \sup_{0 \le x \le 1} \left| \Pr (\{ S_k \alpha \} \le x) -x \right| \]
denote the rate of weak convergence to the uniform distribution in the Kolmogorov metric, where $\{ \cdot \}$ denotes fractional part. Improving results of Schatte \cite{SCH1}, Diaconis \cite[Section 3.C]{DI}, Su \cite{SU} and Hensley and Su \cite{HS}, in \cite{BB2} we found the precise relationship between $\psi (k)$ and the Diophantine approximation properties of $\alpha$. In particular, assuming that $\inf_{h \in \mathbb{Z}\backslash \{ 0 \}} |h| \cdot \| h \alpha \| >0$, where $\| \cdot \|$ denotes distance from the nearest integer (such irrationals are called badly approximable) and that $\E X_1^2 < \infty$, we proved that $k^{-1/2} \ll \psi (k) \ll k^{-1/2}$. See also \cite{BB2} for sharp results under more general Diophantine conditions, and for heavy-tailed $X_1$; for instance, assuming that $\inf_{h \in \mathbb{Z}\backslash \{ 0 \}} |h|^{\gamma} \| h \alpha \| >0$ with some $\gamma \ge 1$, we have $\psi (k) \ll k^{-1/(2 \gamma)}$, and this is also sharp. For similar results on the real line instead of the circle group we refer to Bobkov \cite{BO1}, \cite{BO2}.

In the rational case it is natural to define the distance from uniformity in the ``discrete Kolmogorov metric'' as
\begin{equation}\label{psidisc}
\psi_{\mathrm{disc}} (k) := \max_{0 \le a < q} \left| \Pr (\{ S_k p/q \} \le a/q ) - (a+1)/q \right| ,
\end{equation}
where the maximum is over integer values of $a$. Note that the discrete Kolmogorov metric metrizes weak convergence on $\mathbb{Z}_q$, and so $\psi_{\mathrm{disc}}(k) \to 0$ if and only if the maximal span of $X_1$ (the largest integer $D$ such that $D \mid (X_1-X_2)$ a.s.) is relatively prime to $q$. Our first result establishes an unexpected behavior in the discrete setting: a transition from polynomial to exponential decay. This behavior has not yet been described for a random walk on any compact group in any probability metric.
\begin{thm}\label{discreteKolmogorovtheorem} Assume that $\E X_1^2 < \infty$, and let $\varphi (x) = \E e^{ixX_1}$ denote the characteristic function of $X_1$. There exists a constant $q_0$ depending only on the distribution of $X_1$ with the following property. If $p/q$ is a reduced fraction such that $q \ge q_0$, the maximal span $D$ of $X_1$ and $q$ are relatively prime, and $\min_{0<|h| \le q/2} |h| \cdot \| h p/q \| \ge A>0$ with some constant $A>0$, then
\[ k^{-1/2} \ll \psi_{\mathrm{disc}} (k) \ll k^{-1/2} \qquad \textrm{if } k \le q^2, \]
and
\[ \frac{|\varphi (2 \pi /(Dq))|^k}{q} \ll \psi_{\mathrm{disc}} (k) \ll \frac{|\varphi (2 \pi /(Dq))|^k}{q} \qquad \textrm{if } k>q^2. \]
The implied constants in the lower bounds depend only on the distribution of $X_1$; the implied constants in the upper bounds depend, in addition, on $A$.
\end{thm}
Note that $|\varphi (2 \pi /(Dq))|^k$ is roughly $\exp (-c k/q^2)$ with $c=2 \pi^2 (\mathrm{Var} \, X_1)/D^2$. The value of $A>0$ depends only on the maximal partial quotient in the continued fraction expansion of $p/q$, but not on its length. In particular, if $p/q$ is a best rational approximation to a given badly approximable irrational $\alpha$, this condition is satisfied with an $A>0$ depending only on $\alpha$. In this case, for the first few steps the Markov chains $S_k \alpha \pmod{\mathbb{Z}}$ and $S_k p/q \pmod{\mathbb{Z}}$ are practically indistinguishable. Theorem \ref{discreteKolmogorovtheorem} describes with striking precision, that it takes constant times\footnote{The rates $k^{-1/2}$ and $|\varphi (2 \pi /(Dq))|^k/q$ are equal up to a constant factor if $q^2 \ll k \ll q^2$.} $q^2$ steps for $S_k p/q \pmod{\mathbb{Z}}$ to start to behave like a Markov chain on a finite state space; exponentially fast convergence to the stationary distribution is a hallmark property of finite state space Markov chains. We have an explicitly described transition from rational to irrational behavior: as $p/q \to \alpha$, the time of transition, constant times $q^2$ goes to infinity; in the limiting case of $S_k \alpha \pmod{\mathbb{Z}}$ we end up with the polynomial decay $k^{-1/2}$ for all $k \ge 1$. Obvious modifications of the proof of Theorem \ref{discreteKolmogorovtheorem} yield similar results under more general Diophantine conditions on $p/q$, and for heavy-tailed $X_1$.

In contrast, for any $q \ge q_0$ the distance $\psi_{\mathrm{TV}}(k)$ of the distribution of $\{ S_k p/q \}$ from uniformity in the total variation metric satisfies
\begin{equation}\label{psiTV}
|\varphi (2 \pi /(Dq))|^k \ll \psi_{\mathrm{TV}}(k) \ll |\varphi (2 \pi /(Dq))|^k \qquad \textrm{for all } k \ge 1,
\end{equation}
where $q_0$ and the implied constants depend only on the distribution of $X_1$, provided that $\E X_1^2 < \infty$ and that $D$ and $q$ are relatively prime. The rate does not depend on the Diophantine properties of the fraction $p/q$; indeed, multiplication by an integer $p$ relatively prime to $q$ is a bijection of $\mathbb{Z}_q$, and thus does not change the distance in total variation. In particular, it takes constant times $q^2$ steps to get close to uniformity in the total variation metric as well; however, there is no transition from polynomial to exponential decay. Note that in the irrational setting $S_k \alpha \pmod{\mathbb{Z}}$ does not converge to the uniform distribution on $\mathbb{R}/\mathbb{Z}$ in total variation. Special cases of \eqref{psiTV} were first proved by Chung, Diaconis and Graham \cite{CDG}, see also \cite{AD} and \cite[Section 3.C]{DI}.

In addition to weak convergence, we also study the empirical distribution by considering additive functionals, that is, sums of the form $\sum_{k=1}^N f(S_k \alpha)$. Let\footnote{Throughout the paper we identify $\mathbb{R}/\mathbb{Z}$ by $[0,1)$, and functions on the circle group by $1$-periodic functions on the real line.} $\mathcal{F}$ denote the set of all $1$-periodic functions $f: \mathbb{R} \to \mathbb{R}$ such that $f$ is of bounded variation on $[0,1]$, let $E(f)=\int_0^1 f(x) \, \mathrm{d}x$, and set
\begin{equation}\label{C}
C(\alpha, f) = \E \bar{f}(U)^2 + 2 \sum_{k=1}^{\infty} \E \bar{f}(U) \bar{f}(U+S_k\alpha ),
\end{equation}
where $\bar{f}=f-E(f)$, and $U$ is a random variable uniformly distributed on $\mathbb{R}/\mathbb{Z}$, independent of $X_1, X_2, \dots$. Our second result is the central limit theorem (CLT) and the law of the iterated logarithm (LIL) for the sum $\sum_{k=1}^N f(S_k \alpha )$.
\begin{thm}\label{maintheorem} Assume that $X_1>0$ a.s., and that $\E X_1 < \infty$. Let $\alpha$ be irrational such that $\inf_{h \in \mathbb{Z}\backslash \{0 \}} |h|^{\gamma} \| h \alpha \| >0$ with some constant $1 \le \gamma < (3+\sqrt{5})/4 \approx 1.309$. Then for any $f \in \mathcal{F}$ the series in \eqref{C} is convergent, and $C(\alpha, f) \ge 0$ with equality if and only if $\bar{f}=0$ a.e. For any $f \in \mathcal{F}$ the sum $\sum_{k=1}^N f(S_k \alpha)$ satisfies the central limit theorem
\begin{equation}\label{CLT}
\frac{\sum_{k=1}^N f(S_k \alpha ) - E(f)N}{\sqrt{N}} \overset{d}{\to} \mathcal{N} (0,\sigma^2 )
\end{equation}
and the law of the iterated logarithm
\begin{equation}\label{LIL}
\limsup_{N \to \infty} \frac{\sum_{k=1}^N f(S_k \alpha ) -E(f)N}{\sqrt{2 N \log \log N}} = \sigma \qquad \textrm{a.s.}
\end{equation}
with $\sigma = \sqrt{C(\alpha, f)}$.
\end{thm}
\noindent Note that \eqref{CLT} expresses convergence in distribution to the mean zero normal distribution with variance $\sigma^2$ (interpreted as the constant $0$ if $\sigma =0$). In addition to badly approximable irrationals, the Diophantine condition is satisfied by all algebraic irrationals and also by a.e.\ real number in the sense of the Lebesgue measure. We also prove an almost sure approximation of the same sum by a Wiener process.
\begin{thm}\label{Wienertheorem} Assume that the conditions of Theorem \ref{maintheorem} hold, and let $f \in \mathcal{F}$. After a suitable extension of the probability space, there exists a stochastic process $\zeta (t)$ in the Skorokhod space $D[0, \infty )$ with the same distribution as $\sum_{1 \le k \le t} f(S_k \alpha)$ such that
\[ \zeta (t) = \sigma W(t)+O(t^{1/2-\varepsilon}) \quad \textrm{a.s.} \]
with $\sigma = \sqrt{C(\alpha ,f)}$, a standard Wiener process $W(t)$, and some constant $\varepsilon >0$ depending only on $\gamma$.
\end{thm}
\noindent The almost sure approximation in Theorem \ref{Wienertheorem} immediately implies the CLT and the LIL in Theorem \ref{maintheorem}, as well as the almost sure asymptotics and the limit distribution of more general functionals of the process $\sum_{1 \le k \le t} f(S_k \alpha)$. Using the piecewise linear functions $\sum_{1 \le k \le \lfloor t \rfloor} f(S_k \alpha) + (t-\lfloor t \rfloor ) f(S_{\lfloor t \rfloor +1} \alpha)$ instead, Theorem \ref{Wienertheorem} holds in the space of continuous functions $C[0, \infty )$.

In a previous paper \cite{BB2} we observed a transition in the behavior of the chain from weak to strong dependence as $\gamma$ passes the critical value $2$; the sequence $S_k \alpha \pmod{\mathbb{Z}}$ behaves, from the point of view of discrepancy, like independent random variables if $1 \le \gamma <2$, but not when $\gamma >2$, see Section \ref{relatedsection}. For this reason we conjecture that Theorems \ref{maintheorem} and \ref{Wienertheorem} in fact hold for all $1 \le \gamma <2$.

Starting the chain from its stationary distribution corresponds to $U+S_k \alpha \pmod{\mathbb{Z}}$, where $U$ is as in \eqref{C}. This stationary sequence exhibits the same transition from weak to strong dependence: the sum $\sum_{k=1}^N f(U+S_k \alpha)$, $f \in \mathcal{F}$ satisfies the CLT and the LIL if $1 \le \gamma <2$, but not when $\gamma \ge 2$. The same holds in the quenched setting; that is, the sum $\sum_{k=1}^N f(x+S_k \alpha)$, $f \in \mathcal{F}$ satisfies the CLT and the LIL for a.e.\ $x \in \mathbb{R}$ if $1 \le \gamma <2$, but not when $\gamma \ge 2$. A detailed proof of these results will be given in an upcoming paper. The novelty of Theorems \ref{maintheorem} and \ref{Wienertheorem} lies in the fact that the chain is started from a nonstationary distribution; in fact, from the specific point $x=0$ instead of a typical $x \in \mathbb{R}$. Using $x=0$ as starting point makes the problem considerably harder, and requires a blend of arithmetic and Fourier analytic arguments to complement classical methods of probability theory.

The rational case is of course much simpler. Indeed, if the maximal span of $X_1$ is relatively prime to $q$, then $S_kp/q \pmod{\mathbb{Z}}$ is an irreducible Markov chain on $\mathbb{Z}_q$, and from classical theorems for finite state space Markov chains \cite[Chapter 16]{CH} it follows that for any $f \in \mathcal{F}$, the CLT
\begin{equation}\label{rationalCLT}
\frac{\sum_{k=1}^N f(S_k p/q) -E(f)N}{\sqrt{N}} \overset{d}{\to} \mathcal{N}(0,\sigma^2 )
\end{equation}
and the LIL
\begin{equation}\label{rationalLIL}
\limsup_{N \to \infty} \frac{\sum_{k=1}^N f(S_k p/q) - E(f) N}{\sqrt{2N \log \log N}} = \sigma \qquad \mathrm{a.s.}
\end{equation}
hold with $\sigma = \sqrt{C(p/q, f)}$; the only difference is that $E(f)=q^{-1}\sum_{a=0}^{q-1} f(a/q)$ is the average of $f$ on $\mathbb{Z}_q$, and that in the definition of $C(p/q, f)$ the variable $U$ is uniformly distributed on $\mathbb{Z}_q$. In this case $C(p/q, f) \ge 0$ with equality if and only if $\bar{f}=0$ on $\mathbb{Z}_q$.

Informally, Theorem \ref{maintheorem} states that the Markov chain $S_k p/q \pmod{\mathbb{Z}}$, whose state space is finite but of increasing size as $q \to \infty$, begins to behave more and more like the Markov chain $S_k \alpha \pmod{\mathbb{Z}}$ with an irrational $\alpha$. What we can formally say is that key parameters of these Markov chains, such as the expected value and the variance converge:
\[ \frac{1}{q} \sum_{a=0}^{q-1} f (a/q) \to \int_0^1 f(x) \, \mathrm{d}x \qquad \textrm{and} \qquad C(p/q, f) \to C(\alpha, f) \]
as $p/q \to \alpha$ along the sequence of best rational approximations to a given irrational $\alpha$, see Proposition \ref{Cconvergence}. Once again, we have an explicitly described transition from rational to irrational behavior, with Theorem \ref{maintheorem} corresponding to the limiting case.

The rest of the paper is organized as follows. Several previous results related to Theorems \ref{maintheorem} and \ref{Wienertheorem} are listed in Section \ref{relatedsection}. We prove Theorems \ref{maintheorem} and \ref{Wienertheorem} in Section \ref{ergodicsection}, and the remarks made about the rational case at the end of Section \ref{variancesection}. The proof of Theorem \ref{discreteKolmogorovtheorem} and \eqref{psiTV} are given in Section \ref{ratesection}.

\section{Related results}\label{relatedsection}

Given a sequence of i.i.d.\ $\mathbb{R}/\mathbb{Z}$-valued random variables $\zeta_1, \zeta_2, \dots$, the qualitative behavior of the random walk $Z_k=\sum_{j=1}^k \zeta_j \pmod{\mathbb{Z}}$ is fairly straightforward. A classical result of L\'evy \cite{LE} states that $Z_k$ converges in distribution to the uniform distribution on $\mathbb{R}/\mathbb{Z}$ if and only if the distribution of $\zeta_1$ is not supported on a translate of a finite cyclic subgroup. The corresponding ergodic theorem is due to Robbins \cite{RO}: we have
\begin{equation}\label{equidistribution}
\frac{1}{N} \sum_{k=1}^N f(Z_k) \to \int_0^1 f(x) \, \mathrm{d}x \qquad \textrm{a.s.}
\end{equation}
for all continuous functions $f: \mathbb{R}/\mathbb{Z} \to \mathbb{R}$ if and only if the distribution of $\zeta_1$ is not supported on a finite cyclic subgroup. Note that \eqref{equidistribution} expresses weak convergence of the empirical distribution $N^{-1} \sum_{k=1}^N \delta_{Z_k}$, with $\delta$ denoting the Dirac measure; in other words, the equidistribution of the random sequence $Z_k$. Both facts generalize to compact groups.

Quantitative forms of the ergodic theorem \eqref{equidistribution} are more involved. In a recent paper \cite{BB4} we proved that
\[ \limsup_{N \to \infty} \frac{\left| \sum_{k=1}^N e^{2 \pi i Z_k} \right|}{\sqrt{2 N \log \log N}} = \sigma \qquad \textrm{a.s.} \]
holds with some constant $0<\sigma <\infty$ if and only if $\zeta_1$ is nondegenerate. A remarkable perturbation method first used by Schatte \cite{SCH2}, \cite{SCH3} allows more general test functions. Improving results of Schatte, in \cite{BB3} we showed that the sum $\sum_{k=1}^N f(Z_k)$ satisfies the CLT and the LIL for any function $f \in \mathcal{F}$, provided that the rate of weak convergence of $Z_k$ to uniformity in the Kolmogorov metric
\[ \psi (k) := \sup_{0 \le x \le 1} |\Pr ( \{ Z_k \} \le x) -x | \]
satisfies $\psi (k) \ll k^{-(1+\varepsilon)}$ with some $\varepsilon >0$; we conjecture that the assumption on the rate is optimal. Similar results are known for $p$-H\"older functions (resp.\ bounded Borel measurable functions) with the Kolmogorov metric replaced by the $p$-Wasserstein metric (resp.\ total variation metric); in fact, these hold on any compact group \cite{B1}, \cite{B2}.

The assumption $\psi (k) \ll k^{-(1+\varepsilon)}$ holds for a wide class of distributions (e.g.\ under Cram\'er's condition), but not when $\zeta_1$ is a real-valued lattice variable with finite expectation taken mod $\mathbb{Z}$, as in Theorems \ref{maintheorem} and \ref{Wienertheorem}. Indeed, in the latter case by the Markov inequality and the pigeonhole principle the distribution of $Z_k$ has an atom of weight $\gg k^{-1}$, and thus $\psi (k) \gg k^{-1}$; under a finite variance condition we even have $\psi (k) \gg k^{-1/2}$. The main goal of this paper is to establish quantitative ergodic theorems for such random walks. Instead of fast enough convergence in the Kolmogorov metric, the crucial assumption is $\E |X_1| < \infty$ and $\E X_1 \neq 0$; for technical reasons in Theorems \ref{maintheorem} and \ref{Wienertheorem} we assume the slightly stronger condition $X_1>0$ a.s.\ and $\E X_1 <\infty$. Proving the CLT \eqref{CLT} and the LIL \eqref{LIL} solely under a condition on the expected value presents considerable arithmetic and analytic difficulties; the proof of Theorems \ref{maintheorem} and \ref{Wienertheorem} is much more technical than under $\psi (k)\ll k^{-(1+\varepsilon)}$ in \cite{BB3}. We do not know whether these methods generalize to other compact groups.

From a broader perspective, our results fit into the subject of subsequences $\{ n_k \alpha \}$ of the classical $\{ n \alpha \}$ sequence. For given $n_k$ and $\alpha$, quantitative equidistribution results are notoriously difficult to prove, and are known only in very special cases. Considerable effort has been made to understand the case of a randomly chosen $\alpha$. R.\ Baker \cite{BA} showed that for any strictly increasing sequence of positive integers $n_k$, the mod $1$ discrepancy
\[ D_N (n_k \alpha ) = \sup_{0 \le a<b \le 1} \left| \sum_{k=1}^N I_{[a,b]} (\{ n_k \alpha \}) - N(b-a) \right| \]
satisfies $D_N (n_k \alpha ) \ll N^{1/2} (\log N)^{3/2+\varepsilon}$ for a.e.\ $\alpha$; this is known to be sharp up to factors of $\log N$. For a fixed $f \in \mathcal{F}$ with $\int_0^1 f(x) \, \mathrm{d}x =0$ Lewko and Radziwi\l\l \cite{LR} have recently improved this to
\begin{equation}\label{Lewko}
\left| \sum_{k=1}^N f(n_k \alpha ) \right| \ll (N \log N)^{1/2} (\log \log N)^{3/2+\varepsilon} \qquad \textrm{for a.e. } \alpha ,
\end{equation}
which is known to be sharp up to factors of $\log \log N$; see also \cite{ABS} and \cite{BW}. For lacunary sequences $n_k$ both $D_N (n_k \alpha )$ and $\sum_{k=1}^N f(n_k \alpha )$ with a fixed $f \in \mathcal{F}$ satisfy a sharp LIL for a.e.\ $\alpha$, see Philipp \cite{PH}.

In contrast, our results concern the case when $n_k$ is random and $\alpha$ is deterministic; note that under $X_1>0$ a.s., $\{ S_k \alpha \}$ is a random subsequence of $\{ n \alpha \}$. In a recent paper \cite{BB2} we found the discrepancy $D_N(S_k \alpha )$ up to logarithmic factors for a large class of distributions; an analogue of Baker's theorem. Theorems \ref{maintheorem} and \ref{Wienertheorem} in the present paper represent an improvement for a fixed $f \in \mathcal{F}$ similar to \eqref{Lewko}. In particular, for certain special distributions with $\E |X_1|< \infty$, $\E X_1 \neq 0$ and an irrational $\alpha$ satisfying $0< \liminf_{h \to \infty} h^{\gamma} \| h \alpha \| < \infty$ in \cite{BB2} we showed that $D_N (S_k \alpha)$ is, up to logarithmic factors, $N^{\max \{ 1/2, 1-1/\gamma \}}$ a.s. We thus have a transition from weak dependence to strong dependence at $\gamma =2$. We also mention that in \cite{BB3} we actually proved that $D_N (Z_k)$ satisfies the LIL and found the nondegenerate limit distribution of $N^{-1/2} D_N (Z_k)$ under the assumption $\psi (k) \ll k^{-(1+\varepsilon)}$; the assumptions of Theorem \ref{maintheorem}, however, seem not to be strong enough to find the sharp asymptotics of the discrepancy.

The growth rate of the integer sequence $S_k$ plays an important role in our setup as well. Indeed, under the condition $X_1>0$ a.s.\ and $\E X_1 < \infty$ of Theorems \ref{maintheorem} and \ref{Wienertheorem}, $S_k$ is a linearly increasing sequence of integers, and we have the CLT \eqref{CLT} and the LIL \eqref{LIL} with a fix $f \in \mathcal{F}$. If $X_1$ has heavy tails $\Pr (|X_1| \ge x) \sim c x^{-\beta}$ with some $0<\beta<1$ and $c>0$ instead, and $\inf_{h \in \mathbb{Z} \backslash \{ 0 \}} |h|^{\gamma} \| h \alpha \| >0$ with some $\gamma <1/\beta$, then by the results in \cite{BB1} we have $\psi (k) \ll k^{-1/(\beta \gamma )} \ll k^{-(1+\varepsilon)}$, and consequently the precise asymptotics of $D_N (S_k \alpha )$ is also known. Note that in this heavy-tailed case $S_k$ grows, in a stochastic sense, roughly at the rate $k^{1/\beta}$. In particular, for a random version of Philipp's LIL for the discrepancy, polynomial growth suffices instead of lacunarity.

\section{Proof of Theorems \ref{maintheorem} and \ref{Wienertheorem}}\label{ergodicsection}

Throughout this section $X_1, X_2, \dots$ is a sequence of i.i.d.\ nondegenerate integer-valued random variables with characteristic function $\varphi (x)=\E e^{ixX_1}$, and $S_k = \sum_{j=1}^k X_j$. Further, $e(x)=e^{2 \pi i x}$, and $\hat{f}(h)=\int_0^1 f(x) e(-hx) \, \mathrm{d}x$, $h \in \mathbb{Z}$ are the Fourier coefficients of $f$. Finally, $V(f)$ denotes the total variation of $f$ on $[0,1]$.

\subsection{A lemma on characteristic functions}

In this section we prove a technical lemma on the characteristic function $\varphi$. Let $\mathrm{supp} \, X_1 =\{ n \in \mathbb{Z} : \Pr (X_1=n)>0 \}$ denote the support of (the distribution of) $X_1$. Further, let $\mathrm{gcd} (A)$ denote the positive greatest common divisor, and $A-A=\{ a-b : a,b \in A \}$ the set of differences of a (finite or infinite) set $A \subseteq \mathbb{Z}$.

\begin{lem}\label{charfunctionlemma} Assume that $\E |X_1|< \infty$ and $\E X_1 \neq 0$, and let $d=\mathrm{gcd}\left( \mathrm{supp} \, X_1 \right)$ and $D=\mathrm{gcd}\left( \mathrm{supp}\, X_1 - \mathrm{supp}\, X_1 \right)$.
\begin{enumerate}
\item[(i)] For any integer $N \ge 1$ and any $x,y \in \mathbb{R}$ such that $d(x-y) \not\in \mathbb{Z}$,
\[ \left| \sum_{k=0}^{N-1} \varphi (2 \pi x)^k \varphi (2 \pi y)^{N-k-1} \right| \ll \frac{1}{\| d(x-y) \|} . \]
In particular, for any $x \in \mathbb{R}$ we have $|1-\varphi (2 \pi x)| \gg \| d x \|$.
\item[(ii)] For any positive integers $B,B'$ both divisible by $D/d$ such that $\min \{ B,B' \} \le \sqrt{B+B'}$, and any $x \in \mathbb{R}$ such that $Dx \not\in \mathbb{Z}$,
\[ \left| \frac{1-\varphi (2 \pi x)^B}{1-\varphi (2 \pi x)^{B+B'}} \right| \ll \min \{ B,B' \} . \]
\end{enumerate}
The implied constants in (i) and (ii) depend only on the distribution of $X_1$.
\end{lem}

\begin{proof} Throughout this proof constants and implied constants depend only on the distribution of $X_1$. Replacing $X_1, X_2, \dots$ by $X_1/d, X_2/d, \dots$, we may assume that $d=1$. In particular, the smallest period of $\varphi (2 \pi x)$ is $1$. Note that $X_1-X_2$ has characteristic function $|\varphi|^2$, and $\mathrm{supp} (X_1-X_2)=\mathrm{supp}\, X_1 - \mathrm{supp}\, X_1$. It is thus easy to deduce that $|\varphi (2 \pi x)|=1$ if and only if $x=n/D$ with some integer $n$. In fact, we have the estimate
\begin{equation}\label{phiatnp2}
1-|\varphi (2 \pi x)| \gg \| Dx \|^2 .
\end{equation}
Indeed, note that
\[ 1-|\varphi (2 \pi x)| \ge \frac{1-|\varphi (2 \pi x)|^2}{2} = \E \sin^2 ( \pi x (X_1-X_2)) . \]
Since $\sin^2 (\pi x (X_1-X_2)) \ge 4 x^2 (X_1-X_2)^2$ whenever $|x(X_1-X_2)| \le 1/2$, we get
\[ 1-|\varphi (2 \pi x)| \ge 4 x^2 \E \left( (X_1-X_2)^2 I_{\left\{ |X_1-X_2| \le 1/(2|x|) \right\}} \right) . \]
Here $\E (X_1-X_2)^2 >0$ (possibly infinite) as $X_1$ is nondegenerate, therefore \eqref{phiatnp2} holds in an open neighborhood of $0$; by periodicity and continuity \eqref{phiatnp2} holds for all $x \in \mathbb{R}$.

Furthermore, there exists an integer $a$ such that $X_1 \equiv a \pmod{D}$ a.s., and the assumption $d=1$ ensures that $a$ and $D$ are relatively prime. Hence for any integer $n$ we have $e(nX_1/D)=\omega^n$ a.s. with the primitive $D$th root of unity $\omega=e(a/D)$. This means that within a period $x \in [0,1)$ the curve $\varphi (2 \pi x)$ touches the unit circle at each $D$th root of unity exactly once; moreover, the derivative at these points is nonzero. Indeed, we have
\begin{equation}\label{phiatnp}
\begin{split} \varphi (2 \pi n/D) &= \E e(nX_1/D)=\omega^n, \\ \varphi' (2 \pi n/D) &= \E \left( i X_1 e(nX_1/D) \right) = i \omega^n \E X_1 \neq 0. \end{split}
\end{equation}

We now prove (i). Note that
\[ \sum_{k=0}^{N-1} \varphi (2 \pi x)^k \varphi (2 \pi y)^{N-k-1} = \left\{ \begin{array}{ll} \frac{\varphi (2 \pi x)^N-\varphi (2 \pi y)^N}{\varphi (2 \pi x)-\varphi (2 \pi y)} & \mathrm{if} \,\,\, \varphi (2 \pi x) \neq \varphi (2 \pi y) , \\ N \varphi (2 \pi x)^{N-1} & \mathrm{if} \,\,\, \varphi (2 \pi x)=\varphi (2 \pi y) . \end{array} \right. \]
The derivative $\varphi'$ is uniformly continuous, which in turn ensures that the convergence
\[ \frac{\varphi (2 \pi x)-\varphi (2 \pi y)}{2 \pi x - 2 \pi y} \to \varphi' (2 \pi x) \]
is uniform as $|x-y| \to 0$. Thus it is not difficult to see that there exists a constant $\delta>0$ such that for any $x,y \in J=\bigcup_{n \in \mathbb{Z}}(n/D-\delta, n/D+\delta)$ we have
\begin{equation}\label{varphibound}
\left| \varphi (2 \pi x) - \varphi (2 \pi y) \right| \gg \| x-y \| .
\end{equation}
Since $|\varphi (2 \pi x)| \neq 1$ outside $J$, using the compactness of $[0,1]\backslash J$ and the periodicity of $\varphi$ we also have $|\varphi (2 \pi x)| \le 1-r$ for every $x \in \mathbb{R} \backslash J$ with some constant $r>0$. Let $U=\{ x \in \mathbb{R} : |\varphi (2 \pi x)|>1-r/2 \} \subseteq J$.

Now let $x,y \in \mathbb{R}$ be such that $x-y \not\in \mathbb{Z}$. If $x,y \in U$, then by \eqref{varphibound} we have
\begin{equation}\label{xyinU}
\left| \sum_{k=0}^{N-1} \varphi (2 \pi x)^k \varphi (2 \pi y)^{N-k-1} \right| \le \frac{2}{\left| \varphi (2 \pi x) - \varphi (2 \pi y) \right|} \ll \frac{1}{\| x-y \|} .
\end{equation}
If $x,y \not\in U$, then $|\varphi (2 \pi x)| \le 1-r/2$ and $|\varphi (2 \pi y)| \le 1-r/2$, hence
\[ \left| \sum_{k=0}^{N-1} \varphi (2 \pi x)^k \varphi (2 \pi y)^{N-k-1} \right| \le N \left( 1-\frac{r}{2} \right)^{N-1} \ll 1 \le \frac{1}{\| x-y \|}. \]
Finally, suppose that, say, $x \in U$ and $y \not\in U$. If $y \in J$, then \eqref{varphibound} and \eqref{xyinU} still hold. If $y \not\in J$, then $|\varphi (2 \pi x)| >1-r/2$ and $|\varphi (2 \pi y)| \le 1-r$, therefore
\[ \left| \sum_{k=0}^{N-1} \varphi (2 \pi x)^k \varphi (2 \pi y)^{N-k-1} \right| \le \frac{2}{\left| \varphi (2 \pi x) - \varphi (2 \pi y) \right|} < \frac{4}{r} \ll 1 \le \frac{1}{\| x-y \|}. \]
This finishes the proof of the first claim in (i). The second claim in (i) follows e.g.\ from setting $y=0$ and letting $N \to \infty$ in the first claim (although it would not be difficult to give a direct proof).

To prove (ii), first note that $B$ and $B'$ play almost symmetric roles; that is,
\[ \left| \frac{1-\varphi (2 \pi x)^B}{1-\varphi (2 \pi x)^{B+B'}} \right| \le 1+\left| \frac{1-\varphi (2 \pi x)^{B'}}{1-\varphi (2 \pi x)^{B+B'}} \right| . \]
Thus we may assume that $B \le B'$, and so $B \le \sqrt{B+B'}$. From \eqref{phiatnp} we deduce the asymptotics
\begin{equation}\label{phiasymptotics}
\varphi (2 \pi x) = \omega^n  +2 \pi i \omega^n \E X_1 \left( x-\frac{n}{D} \right) \left( 1+o(1) \right) \quad \mathrm{as} \,\,\, x \to \frac{n}{D}.
\end{equation}
Let, say, $K= 3 \pi |\E X_1|$, and recall \eqref{phiatnp2}. There exists a constant $\delta' >0$ such that for any $x \in J'= \bigcup_{n \in \mathbb{Z}}(n/D-\delta', n/D+\delta')$ the $o(1)$ term in \eqref{phiasymptotics} has absolute value less than $1/2$, and so
\begin{equation}\label{phiestimates}
\begin{split} | \varphi (2 \pi x) -\omega^n | &\le K \left| x-\frac{n}{D} \right|, \\ |\varphi (2 \pi x)| & \le 1- L \left( x- \frac{n}{D} \right)^2 , \end{split}
\end{equation}
with some constant $L>0$, where $n=n(x)$ is the integer for which $|x-n/D|< \delta'$. We may assume $K \delta' <1$. As in the proof of (i), for any $x \in \mathbb{R} \backslash J'$ we have $|\varphi (2 \pi x)| \le 1-r'$ with some constant $r'>0$. To proceed, we will also need the simple estimate
\begin{equation}\label{zestimate}
1-z^m=m(1-z)+O(m|1-z|) \qquad (m \in \mathbb{N}, \,\,\, z \in \mathbb{C}, \,\,\, |1-z|\le 1/m) ,
\end{equation}
which holds with implied constant $e-2$. (This can be seen e.g.\ by expanding $(1+(z-1))^m$.)

Now let $x \in \mathbb{R}$ be such that $Dx \not\in \mathbb{Z}$. If $x \in \mathbb{R} \backslash J'$, then $|1-\varphi (2 \pi x)^{B+B'}| \ge r'$, and we are done. If $\delta'/ \sqrt{B+B'} \le |x-n/p|<\delta'$ for some integer $n$, then from \eqref{phiestimates} we get
\[ \begin{split} |1-\varphi (2 \pi x)^{B+B'}| \ge 1-|\varphi (2 \pi x)|^{B+B'} &\ge 1-\left( 1-L \left( x-\frac{n}{D} \right)^2 \right)^{B+B'} \\ &\ge 1-e^{-L (x-n/D)^2 (B+B')} \\ &\ge 1-e^{-L (\delta' )^2},  \end{split} \]
and again we are done.

Next, assume $\delta'/(B+B') \le |x-n/D|<\delta'/\sqrt{B+B'}$ for some integer $n$. From \eqref{phiestimates} we similarly get
\[ |1-\varphi (2 \pi x)^{B+B'}| \ge 1-e^{-L(x-n/D)^2(B+B')} \ge 1-e^{-L\delta' |x-n/D|} \gg \left| x-\frac{n}{D} \right|.  \]
From \eqref{phiestimates} we also have $|\varphi (2 \pi x)-\omega^n| \le K|x-n/D|<K \delta'/\sqrt{B+B'}<1/B$. Hence \eqref{zestimate} with $z=\omega^{-n}\varphi (2 \pi x)$ and $m=B$ gives
\[ |1-\omega^{-nB} \varphi (2 \pi x)^B| \le (e-1)B|1-\omega^{-n}\varphi (2 \pi x)| \ll B \left| x-\frac{n}{D} \right| . \]
Here $\omega^{-nB}=1$ because $B$ was assumed to be divisible by $D$. Therefore we have $|1-\varphi (2 \pi x)^B|/|1-\varphi (2 \pi x)^{B+B'}| \ll B$, and we are done.

Finally, assume $0<|x-n/D|< \delta'/(B+B')$ for some integer $n$. Then \eqref{phiestimates} gives the stronger bound $|\varphi (2 \pi x)-\omega^n| \le K|x-n/D|<1/(B+B')$. Applying \eqref{zestimate} with $z=\omega^{-n}\varphi (2 \pi x)$ and $m=B$, $m=B+B'$, respectively, we get
\[ \begin{split} |1-\omega^{-nB} \varphi (2 \pi x)^B| &\le (e-1)B|1-\omega^{-n}\varphi (2 \pi x)|, \\ |1-\omega^{-n(B+B')} \varphi (2 \pi x)^{B+B'}| &\ge (3-e)(B+B')|1-\omega^{-n}\varphi (2 \pi x)| .  \end{split} \]
Here $\omega^{-nB}=\omega^{-n(B+B')}=1$, therefore $|1-\varphi (2 \pi x)^B|/|1-\varphi (2 \pi x)^{B+B'}| \ll B/(B+B')$, and we are done.
\end{proof}

\subsection{An exponential sum}

In this section we approximate a general sum involving $f$ by an exponential sum.
\begin{lem}\label{koksmaineq}
For any $f \in \mathcal{F}$, any $x_1, x_2, \dots, x_N \in \mathbb{R}$ and any $y \in [-1/2,1/2]$,
\[ \left| \sum_{k=1}^N f(x_k)-\sum_{k=1}^N f(x_k-y) \right| \le V(f) \left( \sup_{\lambda (J)=|y|} \sum_{k=1}^N I_J^*(x_k) - \inf_{\lambda (J)=|y|} \sum_{k=1}^N I_J^*(x_k) \right) , \]
where the supremum and the infimum are over all intervals $J \subseteq \mathbb{R}$ of length $\lambda (J)=|y|$, and $I_J^*(x)=\sum_{n \in \mathbb{Z}}I_J(x+n)$ denotes the indicator of $J$ extended with period $1$.
\end{lem}

\begin{proof} For any $f \in \mathcal{F}$ and any random variables $X$ and $Y$,
\begin{equation}\label{randomkoksma}
|\E f(X) - \E f(Y)| \le V(f) \sup_{0 \le x \le 1} |\Pr (\{ X \} \le x) - \Pr (\{ Y \} \le x)| ,
\end{equation}
where $\{ \cdot \}$ denotes fractional part. This fact is usually stated when the distribution of $X$ is finitely supported with equal weights, and $Y$ is uniformly distributed on $[0,1]$; see Koksma's inequality \cite[p. 143]{KN}. The general case formally follows from integration by parts; for a detailed proof see \cite[Lemma 1]{BB3}.

Let us apply \eqref{randomkoksma} to the random variables $X$ and $Y$ with distribution $1/N \sum_{k=1}^N \delta_{x_k}$ and $1/N \sum_{k=1}^N \delta_{x_k-y}$, respectively. If $0 \le y \le 1/2$, then for any $x \in [0,1]$ we have
\[ \begin{split} |\Pr (\{ X \} \le x)-\Pr (\{ Y \} \le x)| &= \frac{1}{N} \left| \sum_{k=1}^N \left( I_{[0,x]}(\{ x_k \} )-I_{[0,x]}(\{ x_k-y \}) \right) \right| \\ &= \frac{1}{N} \left| \sum_{k=1}^N \left( I_{[0,y)}^*(x_k)-I_{(x,x+y]}^*(x_k) \right) \right| \\ &\le \frac{1}{N} \left( \sup_{\lambda (J)=|y|} \sum_{k=1}^N I_J^*(x_k) - \inf_{\lambda (J)=|y|} \sum_{k=1}^N I_J^*(x_k) \right) . \end{split} \]
A similar argument shows that the same holds if $-1/2 \le y \le 0$, and the claim follows.
\end{proof}

\begin{lem}\label{fourierapprox}
Let $f \in \mathcal{F}$, and let $x_1, x_2, \dots, x_N \in \mathbb{R}$ be such that $\| x_k-x_{\ell} \| \ge r>0$ for all $k \neq \ell$. For any integer $H>1$,
\[ \sum_{k=1}^N f(x_k) = \sum_{|h|<H} \left( 1-\frac{|h|}{H} \right)\hat{f}(h) \sum_{k=1}^N e^{2 \pi i h x_k} + O \left( V(f) \left( \frac{\log H}{rH}+1 \right) \right) \]
with a universal implied constant.
\end{lem}

\begin{proof} Let $F_H(x)=\sum_{|h|<H} \left( 1-|h|/H \right) e(hx)$ denote the Fej\'er kernel, and recall the convolution identity
\[ \sum_{|h|<H} \left( 1-\frac{|h|}{H} \right) \hat{f}(h) e(hx) = \int_{-1/2}^{1/2} f(x-y) F_H (y) \, \mathrm{d}y . \]
Applying this with $x=x_1, x_2, \dots, x_N$ and using the fact that the total integral of $F_H$ on $[-1/2,1/2]$ is $1$, the error term in the claim can be written in the explicit form
\[ \int_{-1/2}^{1/2} \left( \sum_{k=1}^N f(x_k) - \sum_{k=1}^N f(x_k-y) \right) F_H (y) \, \mathrm{d}y . \]
The assumption $\| x_k-x_{\ell} \| \ge r$, $k \neq \ell$ and the pigeonhole principle imply that the periodic extension of any interval of length $|y|$ contains at most $|y|/r +1$ of the points $x_1, x_2, \dots, x_N$. According to Lemma \ref{koksmaineq} we thus have
\[ \left| \sum_{k=1}^N f(x_k)-\sum_{k=1}^N f(x_k-y) \right| \le V(f) \left( \frac{|y|}{r} +1 \right) . \]
The claim then follows from the estimate $\int_{-1/2}^{1/2}|y|F_H(y) \, \mathrm{d}y \ll \log H /H$, which can be seen directly from the definition of $F_H$.
\end{proof}

\subsection{The variance}\label{variancesection}

In this section we prove two lemmas closely related to the variance of $\sum_{k=1}^N f(S_k \alpha)$. In particular, we find the variance of the corresponding stationary process $\sum_{k=1}^N f(U+S_k \alpha)$, and prove the properties of the constant $C(\alpha, f)$. At the end of the section we prove the remarks made in the Introduction about the rational case.

We will need the fact that the distance from the nearest integer function is symmetric and subadditive; that is, $\|-x\|= \| x \|$ and $\| x+y \| \le \| x \| + \| y \|$ for any $x,y \in \mathbb{R}$. Further, we will need the classical Diophantine estimate which states that for any irrational $\alpha$ such that $\inf_{h \in \mathbb{Z}\backslash \{ 0 \}} |h|^{\gamma} \| h \alpha \| >0$ and any integer $H>1$,
\begin{equation}\label{diophestimate}
\sum_{h=1}^H \frac{1}{h \| h \alpha \|} \ll \left\{ \begin{array}{ll} \log^2 H & \mathrm{if} \,\,\, \gamma =1, \\ H^{\gamma -1} & \mathrm{if} \,\,\, \gamma >1 \end{array} \right.
\end{equation}
with implied constants depending only on $\alpha$ and $\gamma$. A similar estimate claims $\sum_{h=1}^H 1/(h^2 \| h \alpha \|^2) \ll \log H$ if $\gamma =1$. For a detailed proof of \eqref{diophestimate} we refer to \cite[Corollary 4.3]{BB2}; for related results on Diophantine sums see \cite[Chapter 2]{KN}.

\begin{lem}\label{expsummoment}
Assume that $\E |X_1|< \infty$ and $\E X_1 \neq 0$, and let $\alpha$ be irrational such that $\inf_{h \in \mathbb{Z}\backslash \{ 0 \}} |h|^{\gamma} \| h \alpha \| >0$ with some constant $\gamma \ge 1$. Further, let $H>1$ be an integer, and let $c_h \in \mathbb{C}$, $0<|h|<H$ be a finite sequence such that $|c_h| \le 1/|h|$ for every $h$. For any integers $M \ge 0$ and $N \ge 1$,
\[ \begin{split} \E \left| \sum_{0<|h|<H} c_h \sum_{k=M+1}^{M+N} e^{2 \pi i h S_k \alpha} \right|^2 = &N \sum_{0<|h|<H} |c_h|^2 \frac{1-|\varphi (2 \pi h \alpha )|^2}{|1-\varphi (2 \pi h \alpha )|^2} \\ &+ \left\{ \begin{array}{ll} O \left( \log^4 H \right) & \mathrm{if} \,\,\, \gamma =1, \\ O \left( H^{2 \gamma -2} \right) & \mathrm{if} \,\,\, \gamma >1 \end{array} \right. \end{split} \]
with implied constants depending only on the distribution of $X_1$, $\alpha$ and $\gamma$.
\end{lem}

\begin{proof} Let $d=\mathrm{gcd} (\mathrm{supp}\, X_1 )$. Replacing $X_1, X_2, \dots$ by $X_1/d, X_2/d, \dots$ and $\alpha$ by $d \alpha$, we may assume that $d=1$. Let $V$ denote the left hand side of the claim. Expanding the square we get
\begin{equation}\label{expandsquare}
V= \sum_{0<|j|,|h|<H} c_j \overline{c_h} \sum_{k, \ell =M+1}^{M+N} \E e(jS_k \alpha -h S_{\ell}\alpha ).
\end{equation}
Here
\begin{equation}\label{jskhsl}
\E e(jS_k \alpha -h S_{\ell}\alpha ) = \left\{ \begin{array}{ll} \varphi (2 \pi (j-h)\alpha )^k \varphi (-2 \pi h \alpha)^{\ell -k} & \mathrm{if} \,\,\, k<\ell , \\ \varphi (2 \pi (j-h) \alpha)^k & \mathrm{if} \,\,\, k=\ell , \\ \varphi (2 \pi (j-h)\alpha )^{\ell} \varphi (2 \pi j \alpha)^{k - \ell} & \mathrm{if} \,\,\, k> \ell . \end{array} \right.
\end{equation}
First, we estimate the off-diagonal terms $j \neq h$ in \eqref{expandsquare}. Let $a=\varphi (-2 \pi h \alpha)$, $b=\varphi (2 \pi j \alpha)$ and $c=\varphi (2 \pi (j-h) \alpha)$. Summing over $M+1 \le k,\ell \le M+N$ we obtain
\[ \begin{split} \sum_{k,\ell =M+1}^{M+N} \mathbb{E}e(j S_k \alpha -h S_{\ell} \alpha ) &= \\ c^{M+1} \Bigg( &\frac{a}{1-a} \cdot \frac{1-c^N}{1-c} - \frac{a}{1-a} \sum_{k=0}^{N-1}c^k a^{N-k-1} \\ &+\frac{1-c^N}{1-c} +\frac{b}{1-b} \cdot \frac{1-c^N}{1-c} - \frac{b}{1-b} \sum_{k=0}^{N-1}c^k b^{N-k-1} \Bigg) . \end{split} \]
Lemma \ref{charfunctionlemma} (i) shows that $\left|\sum_{k=0}^{N-1}c^k a^{N-k-1} \right| \ll 1/\| j \alpha \|$, $\left|\sum_{k=0}^{N-1}c^k b^{N-k-1} \right| \ll 1/\| h \alpha \|$, $1/|1-a| \ll 1/\| h \alpha \|$, $1/|1-b| \ll 1/\| j \alpha \|$ and $1/|1-c| \ll 1/\| (j-h) \alpha \|$. The contribution of all off-diagonal terms in $V$ is thus
\[ \ll \sum_{\substack{0<|j|,|h|<H \\ j \neq h}} \frac{1}{|j| \cdot |h|} \left( \frac{1}{\| j \alpha \| \cdot \| h \alpha \|} + \frac{1}{\| j \alpha \| \cdot \| (j-h) \alpha \|} + \frac{1}{\| h \alpha \| \cdot \| (j-h) \alpha \|} \right) . \]
Let $E_H=\log^4 H$ if $\gamma =1$, and $E_H=H^{2 \gamma -2}$ if $\gamma >1$ denote the error term in the claim. The classical estimate \eqref{diophestimate} shows that here the contribution of $1/(|j| \cdot \| j \alpha \| \cdot |h| \cdot \| h \alpha \|)$ is indeed $O(E_H)$. Note that if $\| j \alpha \| \le \| h \alpha \|/2$, then $\| (j-h) \alpha \| \ge \| h \alpha \|/2$, hence the contribution of all such terms is also $O(E_H)$. A similar claim holds if $\| h \alpha \| \le \| j \alpha \|/2$. Thus it is enough to consider the terms for which $\| j \alpha \|$ and $\| h \alpha \|$ are equal up to a factor of $2$. In particular, we need to estimate
\begin{equation}\label{maxjh}
\sum_{\substack{0<|j|,|h|<H \\ j \neq h}} \frac{1}{|j| \cdot |h|} \cdot \frac{1}{\max \{ \| j \alpha \|, \| h \alpha \| \} \cdot \| (j-h) \alpha \|} .
\end{equation}
This is symmetric in $j,h$, thus it is enough to consider the terms for which, say, $|j| \le |h|$. For any such term $|h| \ge |j-h|/2$. Hence summing over $j$ and $i=j-h \neq 0$ instead of $j$ and $h$, we get from \eqref{diophestimate} that \eqref{maxjh} is
\[ \ll \sum_{0<|j|<H} \sum_{0<|i|<2H-1} \frac{1}{|j| \cdot \| j \alpha \|} \cdot \frac{1}{|i| \cdot \| i \alpha \|} \ll E_H. \]
The total contribution of all off-diagonal terms $j \neq h$ in \eqref{expandsquare} is thus $O(E_H)$.

Next, we estimate the diagonal terms $j=h$ in \eqref{expandsquare}. Using \eqref{jskhsl} and the fact that $\varphi (-2 \pi h \alpha )=\overline{\varphi (2 \pi h \alpha)}$, after some simplification we get
\[ \sum_{k,\ell =M+1}^{M+N} \E e(hS_k \alpha - hS_{\ell} \alpha ) = N \frac{1-|\varphi (2 \pi h \alpha )|^2}{|1-\varphi (2 \pi h \alpha )|^2} + O\left( \frac{1}{|1-\varphi (2 \pi h \alpha )|^2} \right) . \]
Finally, applying Lemma \ref{charfunctionlemma} (i) and the classical estimate \eqref{diophestimate} we get that the total contribution of this error term in \eqref{expandsquare} also satisfies
\[ \sum_{0<|h|<H} |c_h|^2 \frac{1}{|1-\varphi (2 \pi h \alpha )|^2} \ll \sum_{0<|h|<H} \frac{1}{h^2 \| h \alpha \|^2} \ll E_H . \]
\end{proof}

\begin{lem}\label{varianceprop}
Assume that $\E |X_1|< \infty$ and $\E X_1 \neq 0$, and let $U$ be uniformly distributed on $\mathbb{R}/\mathbb{Z}$, independent of $X_1, X_2, \dots$. Further, let $\alpha$ be irrational such that $\inf_{h \in \mathbb{Z}\backslash \{ 0 \}} |h|^{\gamma} \| h \alpha \| >0$ with some constant $1 \le \gamma <2$.
\begin{enumerate}
\item[(i)] For any $f \in \mathcal{F}$ the infinite series in \eqref{C} is convergent, and $C(\alpha, f) \ge 0$ with equality if and only if $\bar{f}=0$ a.e.
\item[(ii)] For any $f \in \mathcal{F}$ and any integer $N \ge 1$,
\[ \begin{split} \E \left( \sum_{k=1}^N f(U+S_k \alpha) -E(f)N \right)^2 = &C(\alpha ,f) N \\ &+ \left\{ \begin{array}{ll} O \left( V(f)^2 \log N \right) & \mathrm{if} \,\,\, \gamma =1, \\ O \left( V(f)^2 N^{2-2/\gamma} \right) & \mathrm{if} \,\,\, 1<\gamma <2 \end{array} \right. \end{split} \]
with implied constants depending only on the distribution of $X_1$, $\alpha$ and $\gamma$.
\end{enumerate}
\end{lem}

\begin{proof} Let $d=\mathrm{gcd} (\mathrm{supp} \, X_1)$ and $D=\mathrm{gcd} (\mathrm{supp} \, X_1 - \mathrm{supp} \, X_1)$. Replacing $X_1, X_2, \dots$ by $X_1/d, X_2/d, \dots$ and $\alpha$ by $d \alpha$, we may assume that $d=1$. We may also assume that $E(f)=0$. Since the variable $U$ is independent of $X_1, X_2, \dots$, we have
\[ \E f(U) f(U+S_k \alpha ) = \E \int_0^1 f(u) f(u+S_k \alpha ) \, \mathrm{d}u = \E g(S_k \alpha ), \]
where $g(x)=\int_0^1 f(u)f(u+x) \, \mathrm{d}u$. Note that $\hat{g}(h)=|\hat{f}(h)|^2$ for any integer $h$, and that $\hat{f}(0)=0$. Further, integration by parts shows that $|\hat{f}(h)|\le V(f)/(2 \pi |h|)$ for any integer $h \neq 0$. In particular, the Fourier series of $g$ is absolutely convergent. Clearly $g$ is continuous, hence the Fourier series of $g$ converges uniformly to $g$. Therefore we can write
\begin{equation}\label{EGskalpha}
\E g(S_k \alpha) = \E \sum_{h \neq 0} \hat{g}(h) e(hS_k \alpha) = \sum_{h \neq 0} |\hat{f}(h)|^2 \varphi (2 \pi h \alpha )^k.
\end{equation}

First, we prove (i). Recall from the proof of Lemma \ref{charfunctionlemma} that $\varphi (2 \pi x)=1$ if and only if $x \in \mathbb{Z}$; also, $|\varphi (2 \pi x)|=1$ if and only if $Dx \in \mathbb{Z}$. In particular, $|\varphi (2 \pi h \alpha )|<1$ for any $h \neq 0$. For any positive integer $K$ we thus have
\begin{equation}\label{sumK}
\sum_{k=1}^K \E f(U)f(U+S_k \alpha ) = \sum_{h \neq 0} |\hat{f}(h)|^2 \frac{\varphi (2 \pi h \alpha )-\varphi (2 \pi h \alpha )^{K+1}}{1-\varphi (2 \pi h \alpha )} .
\end{equation}
Using Lemma \ref{charfunctionlemma} (i), the classical estimate \eqref{diophestimate} and a dyadic decomposition, in the case $1<\gamma <2$ we get
\[ \begin{split} \sum_{h \neq 0} \frac{|\hat{f}(h)|^2}{|1- \varphi (2 \pi h \alpha )|} \ll \sum_{h \neq 0} \frac{V(f)^2}{h^2 \| h \alpha \|} &= \sum_{m=0}^{\infty} \sum_{2^m \le |h| < 2^{m+1}} \frac{V(f)^2}{h^2 \| h \alpha \|} \\ & \ll \sum_{m=0}^{\infty} \frac{V(f)^2}{2^m} 2^{(\gamma -1)m} < \infty . \end{split} \]
A similar estimate holds if $\gamma =1$. We can thus take the limit as $K \to \infty$ in \eqref{sumK} to obtain
\begin{equation}\label{fuguskalpha}
\sum_{k=1}^{\infty} \E f(U) f(U+S_k \alpha ) = \sum_{h \neq 0} |\hat{f}(h)|^2 \frac{\varphi (2 \pi h \alpha )}{1-\varphi (2 \pi h \alpha )} ,
\end{equation}
and the convergence of the series in \eqref{C} follows. Combining the $h$ and $-h$ terms in \eqref{fuguskalpha} and using $\E f(U)^2=\sum_{h \neq 0} |\hat{f}(h)|^2$, we get
\begin{equation}\label{Fouriervariance}
C(\alpha, f)= \E \bar{f}(U)^2 + 2 \sum_{k=1}^{\infty} \E \bar{f}(U)\bar{f}(U+S_k \alpha ) = \sum_{h \neq 0} |\hat{f}(h)|^2 \frac{1-|\varphi (2 \pi h \alpha)|^2}{|1-\varphi (2 \pi h \alpha )|^2} .
\end{equation}
Here $0<(1-|\varphi (2 \pi h \alpha )|^2)/|1-\varphi (2 \pi h \alpha)|^2 \ll 1/\| h \alpha \|$, and thus $C(\alpha ,f) \ge 0$. Further, $C(\alpha ,f)=0$ if and only if $\hat{f}(h)=0$ for all integers $h \neq 0$; the latter condition is equivalent to $\bar{f}=0$ a.e.

Next, we prove (ii). We may assume $V(f)=1$. Let $V=\E \left( \sum_{k=1}^N f(U+S_k \alpha) \right)^2$. Expanding the square we get
\begin{equation}\label{fusksecondmoment}
V = \sum_{k=1}^N \E f(U+S_k \alpha )^2 +2 \sum_{1 \le k<\ell \le N} \E f(U+S_k \alpha ) f(U+S_{\ell} \alpha ) .
\end{equation}
Let us write $U+S_{\ell} \alpha =U+S_k \alpha + (S_{\ell}-S_k) \alpha$. Since $U+S_k \alpha \pmod{\mathbb{Z}}$ is uniformly distributed on $\mathbb{R}/\mathbb{Z}$ and independent of $(S_{\ell}-S_k) \alpha$, we have $\E f(U+S_k \alpha)^2 = \E f(U)^2 = \| f \|_2^2$ and $\E f(U+S_k \alpha ) f(U+S_{\ell} \alpha ) = \E f(U)f(U+S_{\ell-k} \alpha)$. Hence \eqref{fusksecondmoment} simplifies as $V=N \| f \|_2^2 + 2 \sum_{k=1}^{N-1} (N-k) \E f(U) f(U+S_k \alpha )$. Using \eqref{EGskalpha} we thus have
\[ V=N \| f \|_2^2 + 2 \sum_{h \neq 0} |\hat{f}(h)|^2 \sum_{k=1}^{N-1} (N-k) \varphi (2 \pi h \alpha )^k. \]
Here the inner sum is
\[ \sum_{k=1}^{N-1} (N-k)\varphi (2 \pi h \alpha )^k = N \frac{\varphi (2 \pi h \alpha )}{1-\varphi (2 \pi h \alpha )} - \frac{\varphi (2 \pi h \alpha )-\varphi (2 \pi h \alpha )^{N+1}}{(1-\varphi (2 \pi h \alpha ))^2} . \]
By Lemma \ref{charfunctionlemma} (i) the second term on the right hand side of the previous line has absolute value at most
\[ \min \left\{ \frac{2}{\left| 1-\varphi (2 \pi h \alpha ) \right|^2}, \frac{N}{|1-\varphi (2 \pi h \alpha )|} \right\} \ll \min \left\{ \frac{1}{\| h \alpha \|^2} , \frac{N}{\| h \alpha \|} \right\} . \]
From $|\hat{f}(h)| \le 1/|h|$, the classical estimate \eqref{diophestimate} and a dyadic decomposition, in the case $1<\gamma <2$ we get
\[ \begin{split} \sum_{h \neq 0} |\hat{f}(h)|^2 \min \left\{ \frac{1}{\| h \alpha \|^2} , \frac{N}{\| h \alpha \|} \right\} &\ll \sum_{1 \le h \le 2 N^{1/\gamma}} \frac{1}{h^2 \| h \alpha \|^2} + \sum_{m \ge \log_2 N^{1/\gamma}} \sum_{2^m \le h < 2^{m+1}} \frac{N}{h^2 \| h \alpha \|} \\ &\ll N^{2-2/\gamma} + \sum_{m \ge \log_2 N^{1/\gamma}} N 2^{(\gamma -2)m} \\ &\ll N^{2-2/\gamma} , \end{split} \]
therefore $V=N \| f \|_2^2 + 2 N \sum_{h \neq 0} |\hat{f}(h)|^2 \varphi (2 \pi h \alpha )/(1-\varphi (2 \pi h \alpha )) + O(N^{2-2/\gamma})$. In the case $\gamma =1$ the same holds with error term $O(\log N)$. Combining the $h$ and $-h$ terms, and using $\| f \|_2^2 = \sum_{h \neq 0} |\hat{f}(h)|^2$, we finally obtain
\[ V=N \sum_{h \neq 0} |\hat{f}(h)|^2 \frac{1-|\varphi (2 \pi h \alpha )|^2}{|1-\varphi (2 \pi h \alpha )|^2}+ \left\{ \begin{array}{ll} O \left( V(f)^2 \log N \right) & \mathrm{if} \,\,\, \gamma =1, \\ O \left( V(f)^2 N^{2-2/\gamma} \right) & \mathrm{if} \,\,\, 1<\gamma <2, \end{array} \right. \]
as claimed.
\end{proof}

\begin{remark} If $\E X_1^2 < \infty$, then $\lim_{x \to 0} (1-|\varphi (2 \pi x)|^2)/|1-\varphi (2 \pi x)|^2 = \mathrm{Var} X_1 / (\E X_1)^2$ is finite; in particular, $(1-|\varphi (2 \pi h \alpha)|^2)/|1-\varphi (2 \pi h \alpha)|^2 \ll 1$, and hence $C(\alpha ,f) \ll \| \bar{f} \|_2^2$. Also, the zeroes of the function above are at points $x$ such that $Dx \in \mathbb{Z}$ but $x \not\in \mathbb{Z}$. In the special case $D=1$ we thus also have $(1-|\varphi (2 \pi h \alpha)|^2)/|1-\varphi (2 \pi h \alpha)|^2 \gg 1$, and so $C(\alpha ,f) \gg \| \bar{f} \|_2^2$.
\end{remark}

We now prove the remarks made in the Introduction about the rational case. Let $p/q$ be a reduced fraction, and assume that $D= \mathrm{gcd} (\mathrm{supp}\, X_1 - \mathrm{supp}\, X_1)$ is relatively prime to $q$. Let $\hat{f_q}(h)=q^{-1}\sum_{a=0}^{q-1} f(a/q)e(-ha/q)$, $h=0,1, \dots, q-1$ be the Fourier coefficients of $f$ on $\mathbb{Z}_q=\{ 0,1/q, \dots, (q-1)/q \}$, and let $E(f)=\hat{f_q}(0)$ and $\bar{f}=f-E(f)$. Finally, let $U$ be a random variable uniformly distributed on $\mathbb{Z}_q$, independent of $X_1, X_2, \dots$, and define
\[ C(p/q, f) = \E \bar{f}(U)^2 + 2 \sum_{k=1}^{\infty} \E \bar{f}(U) \bar{f}(U+S_kp/q). \]
Following the proof of Lemma \ref{varianceprop} (i) using Fourier analysis on $\mathbb{Z}_q$ instead of $\mathbb{R}/\mathbb{Z}$, we get
\begin{equation}\label{Fouriervariancerational}
C(p/q, f) = \sum_{h=1}^{q-1} |\hat{f_q}(h)|^2 \frac{1-|\varphi (2 \pi h p/q)|^2}{|1-\varphi (2 \pi h p/q)|^2} .
\end{equation}
In particular, $C(p/q, f) \ge 0$ with equality if and only if $\bar{f}=0$ on $\mathbb{Z}_q$. For a far reaching generalization of these Fourier analytic expressions for $C(\alpha, f)$ and $C(p/q, f)$ to compact groups we refer to \cite{B1}. We also note that the condition $\mathrm{gcd} (D,q)=1$ is equivalent to the distribution of $X_1 p/q \pmod{\mathbb{Z}}$ not being supported on a translate of a proper subgroup of $\mathbb{Z}_q$; consequently, the CLT \eqref{rationalCLT} and the LIL \eqref{rationalLIL} are special cases of the quantitative ergodic theorems for random walks on compact groups in the same paper.

\begin{prop}\label{Cconvergence} Assume that $\E |X_1| < \infty$ and $\E X_1 \neq 0$, and let $\alpha$ be irrational such that $\inf_{h \in \mathbb{Z}\backslash \{ 0 \}} |h|^{\gamma} \| h \alpha \| >0$ with some constant $1 \le \gamma <2$. For any $f \in \mathcal{F}$,
\[ C(p/q, f) \to C(\alpha, f) \]
as $p/q \to \alpha$ along the sequence of best rational approximations, provided that $D= \mathrm{gcd} (\mathrm{supp}\, X_1 - \mathrm{supp}\, X_1)$ is relatively prime to all but finitely many $q$ in this sequence.
\end{prop}

\begin{proof} By \eqref{Fouriervariance} and \eqref{Fouriervariancerational}, we need to prove that
\begin{equation}\label{convergence}
\sum_{\substack{-q/2 < h \le q/2 \\ h \neq 0}} |\hat{f_q}(h)|^2 \frac{1-|\varphi (2 \pi h p/q)|^2}{|1-\varphi (2 \pi h p/q)|^2} \to  \sum_{h \neq 0} |\hat{f}(h)|^2 \frac{1-|\varphi (2 \pi h \alpha)|^2}{|1-\varphi (2 \pi h \alpha )|^2} .
\end{equation}
Since $f$ is Riemann integrable and $\varphi$ is continuous, we have term by term convergence for any fixed $h \neq 0$.

Note that $f$ is of bounded variation also on $\mathbb{Z}_q$ in the sense that
\[ \sum_{a=0}^{q-1} |f((a+1)/q)-f(a/q)| \le V(f). \]
From summation by parts we thus get that
\[ |\hat{f_q}(h)| \le \frac{V(f)}{q |1-e(h/q)|} \ll \frac{V(f)}{|h|} \]
remains true for the $\mathbb{Z}_q$-Fourier transform for all $0<|h| \le q/2$. On the other hand, by Lemma \ref{charfunctionlemma} (i),
\[  0 \le \frac{1-|\varphi (2 \pi h p/q)|^2}{|1-\varphi (2 \pi h p/q)|^2} = 1+\frac{\varphi (2 \pi h p/q)}{1-\varphi (2 \pi h p/q)} + \frac{\overline{\varphi (2 \pi h p/q)}}{1-\overline{\varphi (2 \pi h p/q)}} \ll \frac{1}{\| hp/q \|}. \]
Here
\[ \| hp/q \| \ge \| h \alpha \| - |h| \cdot \left| \alpha -\frac{p}{q} \right| \ge \frac{1}{2} \| h \alpha \| \]
for all $0<|h| \le q/2$; the second inequality follows from the best rational approximation property $\| h \alpha \| \ge \| q \alpha \|$. Therefore
\[ 0 \le |\hat{f_q}(h)|^2 \frac{1-|\varphi (2 \pi h p/q)|^2}{|1-\varphi (2 \pi h p/q)|^2} \ll \frac{V(f)^2}{h^2 \| h \alpha \|} , \]
and, as we have seen before, $\sum_{h \neq 0} 1/(h^2 \| h \alpha \|) < \infty$ is ensured by the assumption $\gamma <2$. The convergence of the series in \eqref{convergence} thus follows e.g.\ from the dominated convergence theorem.
\end{proof}

\subsection{Approximation by independent variables}\label{approxsection}

The main idea of the proof of Theorems \ref{maintheorem} and \ref{Wienertheorem} is that applying suitable small perturbations to the terms of $\sum_{k=1}^N f(S_k \alpha )$ introduces independence; the CLT and the LIL then follow from classical results of probability theory. This method goes back to Schatte \cite{SCH2}, \cite{SCH3}; we have recently improved \cite{BB3}, and generalized his approach to compact groups \cite{B1}, \cite{B2}. In our setup the source of independece is Lemma \ref{schattelemma}. We approximate the error of the perturbations in Lemma \ref{approxerror}, and then finish the proof of Theorems \ref{maintheorem} and \ref{Wienertheorem} at the end of the section.

Fix an irrational $\alpha$ such that $\inf_{h \in \mathbb{Z} \backslash \{ 0 \}} |h|^{\gamma} \| h \alpha \| >0$ with some constant $1 \le \gamma <2$, and let $\psi (k) = \sup_{0 \le x \le 1} |\Pr (\{ S_k \alpha \} \le x) -x|$, as before. For any integer $n \ge 0$, let us decompose the finite set of integers $[2^n,2^{n+1})$ into consecutive blocks $H_{n,1}, J_{n,1}, H_{n,2}, J_{n,2}, \ldots, H_{n,r_n}, J_{n,r_n}$. This way we obtain a block decomposition $H_{n,1}, J_{n,1}, \dots, H_{n,r_n}, J_{n,r_n}$, $n=0,1,\ldots$ of the set of positive integers. Further, let us introduce auxiliary random variables $\xi_{n,2}, \xi_{n,3}, \ldots, \xi_{n,r_n-1}$, $n=0,1,\ldots$ which are independent, uniformly distributed on $\mathbb{R}/\mathbb{Z}$, and independent of $X_1, X_2, \dots$. The following lemma is implicit in Schatte \cite{SCH2}, \cite{SCH3}; for a formal proof see \cite{BB3}.
\begin{lem}[Schatte]\label{schattelemma} There exists a sequence of random variables $\delta_{n,2}, \delta_{n,3}, \dots,$ $\delta_{n,r_n-1}$, $n=0,1,\ldots$ with the following properties. First, $\delta_{n,i}$ is measurable with respect to $X_k$, $k \in J_{n,i-1}$ and $\xi_{n,i}$, and $|\delta_{n,i}|\le \psi (|J_{n,i-1}|)$ for all $n \ge 0$ and $2 \le i<r_n$. Second, the random vectors $(S_k \alpha - \delta_{n,i} \pmod{\mathbb{Z}} \, : \, k \in H_{n,i})$ are independent and have uniformly distributed coordinates. Similarly, there exists a sequence of random variables $\delta_{n,2}', \delta_{n,3}', \dots, \delta_{n,r_n-1}'$, $n=0,1,\ldots$ with the following properties. First, $\delta_{n,i}'$ is measurable with respect to $X_k$, $k \in H_{n,i}$ and $\xi_{n,i}$, and $|\delta_{n,i}'|\le \psi (|H_{n,i}|)$ for all $n \ge 0$ and $2 \le i<r_n$. Second, the random vectors $(S_k \alpha - \delta_{n,i}' \pmod{\mathbb{Z}} \, : \, k \in J_{n,i})$ are independent and have uniformly distributed coordinates.
\end{lem}
As before, let $d=\mathrm{gcd}(\mathrm{supp} \, X_1)$ and $D=\mathrm{gcd} (\mathrm{supp} X_1 - \mathrm{supp} X_1)$. Let $c>0$ be a small constant, to be chosen, and let $B_n=(D/d)^2 \lceil 2^{(1/2-c)n} \rceil$ and $B_n'=(D/d) \lceil 2^{cn} \rceil$. We choose the sizes of the blocks so that $|H_{n,i}|=B_n$ and $|J_{n,i}|=B_n'$ for all $1 \le i < r_n$, and $|H_{n,r_n}|+|J_{n,r_n}|<B_n+B_n'$ is the remainder of $2^n$ modulo $B_n+B_n'$. By the nondegeneracy of $X_1$ we have $\psi (k) \ll k^{-1/(2 \gamma)} \le k^{-1/4}$, see \cite[Proposition 2.1]{BB1}; in particular, $|\delta_{n,i}| \ll (B_n')^{-1/4}$ and $|\delta_{n,i}'| \ll B_n^{-1/4}$.

Now fix $f \in \mathcal{F}$ with $E(f)=\int_0^1 f(x) \, \mathrm{d}x=0$, and consider the block sums $T_{n,i}=\sum_{k \in H_{n,i}} f(S_k \alpha)$ and $D_{n,i}=\sum_{k \in J_{n,i}} f(S_k \alpha)$. If $N=\max J_{n,R}$ for some $n \ge 0$ and $1 \le R \le r_n$, then
\[ \sum_{k=1}^N f(S_k \alpha )= \sum_{m=0}^{n-1} \sum_{i=1}^{r_m} (T_{m,i}+D_{m,i}) + \sum_{i=1}^{R} (T_{n,i}+D_{n,i}) . \]
Let $T_{n,i}^*=\sum_{k \in H_{n,i}} f(S_k \alpha - \delta_{n,i})$ and $D_{n,i}^*=\sum_{k \in J_{n,i}} f(S_k \alpha - \delta_{n,i}')$ denote the corresponding perturbed block sums. By Lemma \ref{schattelemma} we have that $T_{n,i}^*$, $n \ge 0$, $2 \le i <r_n$ are independent and $\E T_{n,i}^*=0$; similarly, $D_{n,i}^*$, $n \ge 0$, $2 \le i <r_n$ are independent and $\E D_{n,i}^*=0$. Further, $T_{n,i}^* \overset{d}{=} \sum_{k=1}^{|H_{n,i}|} f(U+S_k \alpha)$ and $D_{n,i}^* \overset{d}{=} \sum_{k=1}^{|J_{n,i}|} f(U+S_k \alpha)$, therefore by Lemma \ref{varianceprop} the variance is
\begin{equation}\label{tni*dni*variance}
\E T_{n,i}^{*2} = C(\alpha ,f) |H_{n,i}| + \left\{ \begin{array}{ll} O \left( V(f)^2 \log |H_{n,i}| \right) & \mathrm{if} \,\,\, \gamma =1, \\ O \left( V(f)^2 |H_{n,i}|^{2-2/\gamma} \right) & \mathrm{if} \,\,\, 1<\gamma <2 \end{array} \right.
\end{equation}
with implied constants depending only on the distribution of $X_1$, $\alpha$ and $\gamma$; the same holds for $\E D_{n,i}^{*2}$ with $|H_{n,i}|$ replaced by $|J_{n,i}|$.

\begin{lem}\label{approxerror} Assume that $X_1 >0$ a.s.\ and $\E X_1 < \infty$, and also that $1<\gamma<(3+\sqrt{5})/4 \approx 1.309$. There exists a small enough constant $c>0$ depending only on $\gamma$, such that for any $n \ge 0$ and any real $t \ge 1$,
\[ \begin{split} \Pr \left( \max_{2 \le R < r_n} \left| \sum_{i=2}^R (T_{n,i} - T_{n,i}^*) \right| \ge t V(f) n 2^{(1/2-c(2-\gamma )/8)n} \right) &\ll t^{-1/\gamma}, \\ \Pr \left( \max_{2 \le R < r_n} \left| \sum_{i=2}^R (D_{n,i} - D_{n,i}^*) \right| \ge t V(f) n 2^{(1/2-c(2-\gamma )/8)n} \right) &\ll t^{-1/\gamma} \end{split} \]
with implied constants depending only on the distribution of $X_1$, $\alpha$, $\gamma$ and $c$.
\end{lem}

\begin{proof} We only give a proof for $\sum_{i=2}^R (T_{n,i}-T_{n,i}^*)$, as the proof for $\sum_{i=2}^R (D_{n,i}-D_{n,i}^*)$ is analogous. Replacing $X_1, X_2, \dots$ by $X_1/d, X_2/d, \dots$ and $\alpha$ by $d \alpha$, we may assume that $d=1$. We may also assume that $V(f)=1$. Since $n \ge 0$ is fixed, for the sake of simplicity we can write $T_i=T_{n,i}$, $H_i=H_{n,i}$, $J_i=J_{n,i}$, $B=B_n$, $B'=B_n'$, $r=r_n$, $\xi_i=\xi_{n,i}$ and $\delta_i=\delta_{n,i}$. Let $c>0$ be a small enough constant to be chosen, and note that $\min \{ B,B' \}=B' \le \sqrt{B+B'}$. Further, let $H>1$ be an integer, to be chosen. We start by approximating $T_i$ and $T_i^*$ by two exponential sums. For every $2 \le i <r$ let
\[ \begin{split} A_i &= \sum_{0<|h|<H} \left( 1-\frac{|h|}{H} \right) \hat{f}(h) \sum_{k \in H_i} e \left( h(S_k \alpha ) \right), \\ A_i^* &= \sum_{0<|h|<H} \left( 1-\frac{|h|}{H} \right) \hat{f}(h) \sum_{k \in H_i} e\left( h(S_k \alpha -\delta_i ) \right) , \end{split} \]
and let us write $T_i=A_i +E_i$ and $T_i^*=A_i^* +E_i^*$.

Fix $2 \le R < r$, and let us apply Lemma \ref{fourierapprox} to the points $S_k \alpha$, $k \in H_2 \cup \cdots \cup H_R$. Since $X_1$ attains positive integers only, these points satisfy
\[ \left\| S_k \alpha - S_{\ell} \alpha  \right\| \gg \left| S_k-S_{\ell} \right|^{-\gamma} \ge \left( \sum_{j=2^n}^{2^{n+1}} X_j \right)^{-\gamma} \]
for any $k \neq \ell$, and hence $\left| \sum_{i=2}^R E_i \right| \ll \log H/H \left( \sum_{j=2^n}^{2^{n+1}} X_j \right)^{\gamma} +1$. The same holds for the maximum over $2 \le R <r$, thus by $\E X_1<\infty$ and the Markov inequality,
\begin{equation}\label{maxER}
\Pr \left( \max_{2 \le R < r} \left| \sum_{i=2}^R E_i \right| \ge t n 2^{(1/2-c(2-\gamma )/8)n} \right) \ll \left( \frac{\log H}{t H} \right)^{1/\gamma} 2^{(1-1/(2 \gamma ) +c(2-\gamma )/(8 \gamma ))n} .
\end{equation}
The last term in the exponent satisfies $c(2-\gamma )/(8 \gamma ) \le c$.

Using the fact that the vectors $(S_k \alpha - \delta_i \pmod{\mathbb{Z}} \, : \, k \in H_i)$, $2 \le i < r$ are independent and their coordinates are uniformly distributed on $\mathbb{R}/\mathbb{Z}$, we get that $E_i^*$, $2 \le i <r$ are independent and $\E E_i^* =0$. Letting
\[ g(x)=f(x)-\sum_{0<|h|<H} \left( 1-\frac{|h|}{H} \right) \hat{f}(h) e(hx), \]
we have $E_i^*=T_i^*-A_i^* = \sum_{k \in H_i} g(S_k \alpha - \delta_i) \overset{d}{=} \sum_{k=1}^{|H_i|} g(U+S_k \alpha)$. Note that $\hat{g}(h)=|h| \hat{f}(h)/H$ if $0<|h|<H$, and $\hat{g}(h)=\hat{f}(h)$ if $|h| \ge H$. Using the nonnegativity of the Fej\'er kernel it is also not difficult to see that the Ces\`aro mean in the previous line has total variation $\le V(f)$, and hence $V(g) \le 2V(f)=2$. From the results seen in the proof of Lemma \ref{varianceprop} it thus follows that
\[ \begin{split} \E |E_i^*|^2 &\ll C(\alpha ,g) |H_i| + |H_i|^{2-2/\gamma} \\ &= |H_i| \sum_{h \neq 0} |\hat{g}(h)|^2 \frac{1-|\varphi (2 \pi h \alpha )|^2}{|1-\varphi (2 \pi h \alpha )|^2} + |H_i|^{2-2/\gamma} \\ &\ll |H_i| \sum_{h=1}^{H-1} \frac{1}{H^2 \| h \alpha \|} + |H_i| \sum_{h=H}^{\infty} \frac{1}{h^2 \| h \alpha \|} + |H_i|^{2-2/\gamma} \\ &\ll |H_i| H^{\gamma -2} + |H_i|^{2-2/\gamma} , \end{split} \]
and hence by the Kolmogorov inequality, after simplifying the exponent,
\begin{equation}\label{maxER*}
\Pr \left( \max_{2 \le R < r} \left| \sum_{i=2}^R E_i^* \right| \ge t n 2^{(1/2-c (2-\gamma )/8)n} \right) \le \frac{\sum_{i=2}^r \E |E_i^*|^2}{t^2 2^{(1-c (2-\gamma )/4)n}} \ll \frac{2^{cn}}{t^2 H^{2-\gamma}} + \frac{1}{t^2}
\end{equation}
provided that $c>0$ is small enough. Combining \eqref{maxER} and \eqref{maxER*} we can estimate the error of replacing $T_i$ by $A_i$ and $T_i^*$ by $A_i^*$, and we obtain
\begin{equation}\label{yiyi*}
\begin{split} \Pr \Bigg( \max_{2 \le R < r} \left| \sum_{i=2}^R (T_i - T_i^*) \right| &\ge t n 2^{(1/2-c(2-\gamma )/8)n} \Bigg) \ll \\ &\Pr \left( \max_{2 \le R < r} \left| \sum_{i=2}^R (A_i - A_i^*) \right| \ge \frac{t}{2} n 2^{(1/2-c(2-\gamma )/8)n} \right) \\ &+ \left( \frac{\log H}{t H} \right)^{1/\gamma} 2^{(1-1/(2 \gamma ) +c)n} + \frac{2^{cn}}{t^2 H^{2-\gamma}} + \frac{1}{t^2} .
\end{split}
\end{equation}

Next, fix $1 \le R <S<r$, and let us estimate $\E \left| \sum_{i=R+1}^S (A_i-A_i^*)\right|^2$. We start with the diagonal term $\E |A_i-A_i^*|^2$. Let $Y_i=\sum_{k \in J_{i-1}} X_k$. We have
\begin{equation}\label{ai-ai*}
A_i-A_i^*= \sum_{0<|h|<H} \left( 1-\frac{|h|}{H} \right) \hat{f}(h) \left( e(h Y_i \alpha )-e(h (Y_i \alpha - \delta_i )) \right) \sum_{k \in H_i} e(h (S_k-Y_i) \alpha ) ,
\end{equation}
where the factor $(e(h Y_i \alpha ) - e(h (Y_i \alpha - \delta_i))$ is independent of the sum over $k \in H_i$. Let $\mathcal{F}_i$ denote the $\sigma$-algebra generated by $X_k$, $k \in J_{i-1}$ and $\xi_i$, and note that $Y_i$ and $\delta_i$ are $\mathcal{F}_i$-measurable. We can apply Lemma \ref{expsummoment} to the i.i.d.\ sequence obtained from $X_1, X_2, \dots$ by deleting the terms $X_k$, $k \in J_{i-1}$ to estimate the conditional expectation with respect to $\mathcal{F}_i$ as
\begin{equation}\label{conditionalbound}
\begin{split} &\E \big( |A_i-A_i^*|^2 \mid \mathcal{F}_i \big) = \\ &|H_i| \sum_{0<|h|<H} \left( 1-\frac{|h|}{H} \right)^2 |\hat{f}(h)|^2 \cdot \left| 1-e(  -h\delta_i ) \right|^2 \frac{1-|\varphi (2 \pi h \alpha )|^2}{|1-\varphi (2 \pi h \alpha )|^2} +O\left( H^{2 \gamma -2} \right) . \end{split}
\end{equation}
Here $|H_i| \ll 2^{(1/2 -c)n}$, $|\hat{f}(h)|^2 \ll 1/h^2$ and $(1-|\varphi (2 \pi h \alpha)|^2) / |1-\varphi (2 \pi h \alpha)|^2 \ll 1/\| h \alpha \|$. On the one hand, $|1-e(-h \delta_i )|^2 \ll 1$. On the other hand, using $|\delta_i| \ll (B')^{-1/4} \ll 2^{-cn/4}$ we also have $|1-e(-h \delta_i )|^2 \ll h^2 |\delta_i|^2 \ll h^2 2^{-cn/2}$. The main term of \eqref{conditionalbound} is therefore
\[ \ll 2^{(1/2-c)n} \left( \sum_{0<|h|< 2^{cn/4}} \frac{2^{-cn/2}}{\| h \alpha \|} + \sum_{|h| \ge 2^{cn/4}} \frac{1}{h^2 \| h \alpha \|} \right) \ll 2^{(1/2-c-c(2-\gamma )/4)n} . \]
Taking the (total) expectation of \eqref{conditionalbound} and summing over $R+1 \le i \le S$, we obtain
\begin{equation}\label{diagestimate}
\sum_{i=R+1}^S \E |A_i-A_i^*|^2 \ll 2^{(1/2-c-c(2-\gamma )/4)n} (S-R) + H^{2 \gamma -2} (S-R) .
\end{equation}

Now fix the indices $R+1 \le i<j \le S$, and let us estimate the off-diagonal term $\E (A_i-A_i^*)(\overline{A_j}-\overline{A_j^*})$. Let $Y_i=\sum_{k \in J_{i-1}} X_k$ and $Y_j=\sum_{k \in J_{j-1}} X_k$. Note that \eqref{ai-ai*} still holds. We can thus derive a factorization of $A_j-A_j^*$ similar to \eqref{ai-ai*}, from which we altogether get
\begin{equation}\label{ai-ai*aj-aj*}
\begin{split} (A_i-&A_i^*) (\overline{A_j}-\overline{A_j^*}) = \\ &\sum_{\substack{0<|h_1|<H \\ 0<|h_2| <H}} \left( 1-\frac{|h_1|}{H} \right) \left( 1-\frac{|h_2|}{H} \right) \hat{f}(h_1) \overline{\hat{f}(h_2)} \times \\ &\left( e(h_1 Y_i \alpha )-e(h_1 (Y_i \alpha - \delta_i )) \right) \left( e(-h_2 (Y_i+ Y_j)\alpha) -e(-h_2 ((Y_i+Y_j) \alpha -\delta_j )) \right) \times \\ &\sum_{\substack{k \in H_i \\ \ell \in H_j}} e(h_1 (S_k-Y_i) \alpha - h_2 (S_{\ell}-Y_i-Y_j) \alpha ) . \end{split}
\end{equation}
We now take the expected value of \eqref{ai-ai*aj-aj*}. Note that the factor in terms of $Y_i, \delta_i, Y_j, \delta_j$ is independent of the sum over $k, \ell$. Observe also, that
\begin{equation}\label{eh1ai}
\begin{split} c(h_1, h_2) :=\E &\left( e(h_1 Y_i \alpha )-e(h_1 (Y_i \alpha - \delta_i )) \right) \times \\ &\left( e(-h_2 (Y_i+ Y_j)\alpha) -e(-h_2 ((Y_i+Y_j) \alpha -\delta_j )) \right) \end{split}
\end{equation}
depends only on $h_1, h_2$ but not on $i,j$. Indeed, since we chose the sizes of the blocks $J_{i-1}, J_{j-1}$ defining $Y_i, \delta_i, Y_j, \delta_j$ to be the same for all $i,j$ (within the interval $[2^n,2^{n+1})$), the random variable in \eqref{eh1ai} has the same distribution for all $i<j$. We have
\begin{equation}\label{h1tk-h2tl}
\E e(h_1 (S_k-Y_i) \alpha - h_2 (S_{\ell}-Y_i-Y_j) \alpha ) = \varphi (2 \pi (h_1-h_2 ) \alpha )^{k-B'} \varphi (-2 \pi h_2 \alpha )^{\ell -k-B'} .
\end{equation}
Let $a=\varphi (2 \pi (h_1-h_2)\alpha )$ and $b=\varphi (-2 \pi h_2 \alpha )$. Summing \eqref{h1tk-h2tl} over $k \in H_i$ and $\ell \in H_j$ we get
\[ \begin{split} \sum_{\substack{k \in H_i \\ \ell \in H_j}} \E e(h_1 (S_k-Y_i) \alpha - &h_2 (S_{\ell}-Y_i-Y_j) \alpha ) \\ &= \sum_{k=2^n+(i-1)(B+B')}^{2^n+(i-1)(B+B')+B-1} \sum_{\ell =2^n+(j-1)(B+B')}^{2^n+(j-1)(B+B')+B-1} a^{k-B'} b^{\ell -k-B'} \\ &=a^{2^n+(i-2)(B+B')+B} b^{(j-i-1)(B+B')} \frac{1-b^B}{1-b} \sum_{k=0}^{B-1} a^k b^{B-k-1}. \end{split} \]
From \eqref{ai-ai*aj-aj*}, \eqref{eh1ai} and the estimates $|c(h_1,h_2)| \le 4$, $|\hat{f}(h_1)| \le 1/|h_1|$, $|\hat{f}(h_2)| \le 1/|h_2|$ we get by fixing $R+1 \le i < S$ and summing over $j=i+1, i+2, \dots, S$ that
\[ \left| \sum_{j=i+1}^S \E (A_i-A_i^*) (\overline{A_j}-\overline{A_j^*}) \right| \ll \sum_{\substack{0<|h_1|<H \\ 0<|h_2| <H}} \frac{1}{|h_1| \cdot |h_2|} \cdot \left| \frac{1-b^B}{1-b^{B+B'}} \right| \cdot \frac{1}{|1-b|} \cdot \left| \sum_{k=0}^{B-1} a^k b^{B-k-1} \right| . \]
Lemma \ref{charfunctionlemma} (i) shows that here $1/|1-b| \ll 1/\| h_2 \alpha \|$ and $\left| \sum_{k=0}^{B-1} a^k b^{B-k-1} \right| \ll 1/\| h_1 \alpha \|$, while Lemma \ref{charfunctionlemma} (ii) gives $|1-b^B|/|1-b^{B+B'}| \ll 2^{cn}$. Hence
\[ \left| \sum_{j=i+1}^S \E (A_i-A_i^*) (\overline{A_j}-\overline{A_j^*}) \right| \ll 2^{cn} \sum_{\substack{0<|h_1|<H \\ 0<|h_2| <H}} \frac{1}{|h_1| \cdot \| h_1 \alpha \|} \cdot \frac{1}{|h_2| \cdot \| h_2 \alpha \|} . \]
Applying the classical Diophantine estimate \eqref{diophestimate}, and then summing over $R+1 \le i<S$ we can thus estimate the contribution of the off-diagonal terms $i<j$. The terms $j<i$ can be estimated similarly, and we finally get
\begin{equation}\label{offdiagestimate}
\left| \sum_{\substack{R+1 \le i,j \le S \\ i \neq j}} \E (A_i-A_i^*)(\overline{A_j} - \overline{A_j^*}) \right| \ll H^{2 \gamma -2} 2^{cn} (S-R).
\end{equation}
Combining \eqref{diagestimate} and \eqref{offdiagestimate} we get that for any $1 \le R<S <r$,
\[ \E \left| \sum_{i=R+1}^S (A_i-A_i^*) \right|^2 \ll 2^{(1/2-c-c(2-\gamma )/4)n} (S-R) + H^{2 \gamma -2} 2^{cn} (S-R), \]
and thus by the Rademacher--Menshov inequality \cite[Theorem F]{MO},
\[ \E \left( \max_{2 \le R<r} \left| \sum_{i=2}^R (A_i-A_i^*) \right| \right)^2 \ll 2^{(1/2-c-c(2-\gamma )/4)n} r \log^2 r + H^{2 \gamma -2} 2^{cn} r \log^2 r. \]
Note that $r \ll 2^{(1/2+c)n}$. Applying the Chebyshev inequality, from \eqref{yiyi*} we finally deduce
\[ \begin{split} \Pr \Bigg( \max_{2 \le R < r} &\left| \sum_{i=2}^R (T_i - T_i^*) \right| \ge t n 2^{(1/2-c(2-\gamma )/8)n} \Bigg) \ll \\ &\frac{H^{2\gamma -2} 2^{(-1/2+3c)n}}{t^2} + \left( \frac{\log H}{t H} \right)^{1/\gamma} 2^{(1-1/(2 \gamma ) +c)n} + \frac{2^{cn}}{t^2 H^{2-\gamma}} +\frac{1}{t^2} . \end{split} \]
Choosing $H=\lceil 2^{(3 \gamma -1) n /(4 \gamma^2-4\gamma +2)} \rceil$ the first two error terms have roughly the same order of magnitude. The estimate then simplifies to $\ll 2^{(-\tau +3c)n}n t^{-1/\gamma}$ with some $\tau$ depending only on $\gamma$; moreover, we have $\tau >0$ whenever $1<\gamma < (3+\sqrt{5})/4$. Therefore if $c>0$ is small enough depending only on $\gamma$, the estimate is $\ll t^{-1/\gamma}$, as claimed.
\end{proof}

\begin{proof}[Proof of Theorem \ref{maintheorem}] The properties of $C(\alpha, f)$ were proved in Lemma \ref{varianceprop} (i); it remains to prove the CLT \eqref{CLT} and the LIL \eqref{LIL}. We may assume that $E(f)=0$ and $V(f)=1$, and that $1<\gamma < (3+\sqrt{5})/4$. Let us fix a small enough constant $c>0$ for which the claim of Lemma \ref{approxerror} holds. For any integer $N \ge 1$ there exist integers $n=n(N) \ge 0$ and $2 \le R=R(N) <r_n$ such that $2^n \le N < 2^{n+1}$ and $|N-\max J_{n,R}| \ll 2^{(1/2-c)n}$. Further, $T_{m,1}$, $D_{m,1}$, $T_{m,r_m}$ and $D_{m,r_m}$ are all $O(2^{(1/2-c)m})$. Hence
\[ \sum_{m=0}^n (|T_{m,1}|+|D_{m,1}|+|T_{m,r_m}|+|D_{m,r_m}|) \ll 2^{(1/2-c)n} \ll N^{1/2-c} , \]
and consequently
\begin{equation}\label{fSkfirstdecomposition}
\sum_{k=1}^N f(S_k \alpha ) = \sum_{m=0}^{n-1} \sum_{i=2}^{r_m-1} (T_{m,i}+D_{m,i}) + \sum_{i=2}^R (T_{n,i}+D_{n,i}) + O(N^{1/2-c}) .
\end{equation}
Applying Lemma \ref{approxerror} with $t=m^{-1} 2^{\varepsilon m}$ where, say, $\varepsilon =c(2-\gamma )/16$, we get
\[ \Pr \left( \max_{2 \le R <r_m} \left|\sum_{i=2}^{R} (T_{m,i}+D_{m,i} - T_{m,i}^*-D_{m,i}^*) \right| \ge 2^{(1/2-\varepsilon ) m} \right) \ll m^{1/\gamma} 2^{-\varepsilon m /\gamma } , \]
and thus by the Borel--Cantelli lemma
\[ \max_{2 \le R < r_m} \left|\sum_{i=2}^{R} (T_{m,i}+D_{m,i} - T_{m,i}^*-D_{m,i}^*)\right|\ll 2^{(1/2-\varepsilon )m} \quad \text{a.s.}  \]
Summing over $m=0, 1, \ldots, n-1$, we see that replacing $T_{m,i}$ by $T_{m,i}^*$ and $D_{m,i}$ by $D_{m,i}^*$, the double sum on the right hand side of  \eqref{fSkfirstdecomposition} changes by $O(2^{(1/2-\varepsilon )n})=O(N^{1/2-\varepsilon})$. The same holds if we replace $T_{n,i}$ by $T_{n,i}^*$ and $D_{n,i}$ by $D_{n,i}^*$ in the second sum on the right hand side of \eqref{fSkfirstdecomposition}, and so we get
\[ \sum_{k=1}^{N} f(S_k \alpha) = \sum_{m=0}^{n-1} \sum_{i=2}^{r_m-1} (T_{m,i}^*+D_{m,i}^*) + \sum_{i=2}^{R} (T_{n,i}^*+D_{n,i}^*) + O(N^{1/2-\varepsilon}) \quad \text{a.s.} \]
Recall that the variables $D_{m, i}^*$, $2\le i < r_m$, $m=0, 1, \ldots$, viewed as a single sequence, are independent, mean zero random variables with variance $\E D_{m,i}^{*2} \ll |J_{m,i}| \ll 2^{cm}$, see \eqref{tni*dni*variance}. By the strong law of large numbers, the contribution of $D_{m,i}^*$ is negligible:
\[ \sum_{m=0}^{n-1} \sum_{i=2}^{r_m-1} D_{m,i}^* + \sum_{i=2}^R D_{n,i}^* \ll 2^{(1/4+c+\varepsilon )n} \ll N^{1/2-\varepsilon} \quad \textrm{a.s.}, \]
and so
\begin{equation}\label{fSkseconddecomposition}
\sum_{k=1}^{N} f(S_k \alpha) = \sum_{m=0}^{n-1} \sum_{i=2}^{r_m-1} T_{m,i}^* + \sum_{i=2}^{R} T_{n,i}^* + O(N^{1/2-\varepsilon}) \quad \text{a.s.}
\end{equation}
Here $T_{m, i}^*$, $2\le i < r_m$, $m=0, 1, \ldots$, viewed as a single sequence, are also independent, mean zero random variables with variance $\E T_{m,i}^{*2} \sim C( \alpha ,f) |H_{m,i}|$, see \eqref{tni*dni*variance}. Since $\sum_{m=0}^{n-1} \sum_{i=2}^{r_m-1} |H_{m,i}| + \sum_{i=2}^R |H_{n,i}| \sim N$, and $|T_{m,i}^*|=O(N^{1/2-c})$, $|T_{n,i}^*|=O(N^{1/2-c})$, the Lindeberg condition \cite[p.\ 292]{LO}, as well as Kolmogorov's condition for the LIL \cite[p.\ 272]{LO} are satisfied, and consequently
\[ N^{-1/2} \left( \sum_{m=0}^{n-1} \sum_{i=2}^{r_m-1} T_{m,i}^* + \sum_{i=2}^R T_{n,i}^* \right) \overset{d}{\to} \mathcal{N} \left( 0, \sigma^2 \right) \]
and
\[ \limsup_{N \to \infty} \left( 2 N \log \log N \right)^{-1/2} \left( \sum_{m=0}^{n-1} \sum_{i=2}^{r_m-1} T_{m,i}^* + \sum_{i=2}^R T_{n,i}^* \right) = \sigma \quad \text{a.s.} \]
with $\sigma = \sqrt{C(\alpha, f)}$. By \eqref{fSkseconddecomposition}, the CLT \eqref{CLT} and the LIL \eqref{LIL} follow.
\end{proof}

\begin{proof}[Proof of Theorem \ref{Wienertheorem}] We may assume that $E(f)=0$ and $V(f)=1$, and that $1<\gamma <(3+\sqrt{5})/4$. In the proof of Theorem \ref{maintheorem} we wrote $\sum_{1 \le k \le t} f(S_k \alpha )$ in the form $\sum_{1 \le k \le t} f(S_k \alpha ) = X(t)+Y(t)$ with stochastic processes
\[ X(t)= \sum_{m=0}^{n(t)-1} \sum_{i=2}^{r_m-1} T_{m,i}^* + \sum_{i=2}^{R(t)} T_{n(t),i}^* \]
and $Y(t)=O(t^{1/2-\varepsilon})$ a.s., where $\varepsilon >0$ is a constant depending only on $\gamma$, and $n(t) \ge 0$ and $2 \le R(t) < r_{n(t)}$ are integers such that $|t-\max J_{n(t),R(t)}| \ll t^{1/2-\varepsilon}$. Note that $X(t)$ and $Y(t)$ are measurable functions of the variables $X_k$ and the auxiliary variables $\xi_{n,i}$.

Applying a theorem of Strassen on the almost sure approximation of sums of independent variables by a Wiener process \cite[Theorem 4.4]{ST}, we get that after a suitable extension of the probability space there exists a stochastic process $\tilde{X}(t)$ in the Skorokhod space $D[0,\infty )$ with the same distribution as $X(t)$, such that $\tilde{X}(t)=\sigma W(t) + O(t^{1/2-\varepsilon'})$ a.s.\ with $\sigma =\sqrt{C(\alpha, f)}$, a standard Wiener process $W(t)$ and some constant $\varepsilon'>0$ depending only on $\gamma$. After another extension of the probability space we can introduce independent variables $\tilde{X}_k \overset{d}{=} X_k$ and $\tilde{\xi}_{n,i} \overset{d}{=} \xi_{n,i}$ such that $\tilde{X}(t)$ is the same measurable function of the $\tilde{X}_k$ and $\tilde{\xi}_{n,i}$ as $X(t)$ is of the $X_k$ and $\xi_{n,i}$. Let $\tilde{Y}(t)\overset{d}{=} Y(t)$ be the same measurable function of the $\tilde{X}_k$ and $\tilde{\xi}_{n,i}$ as $Y(t)$ is of the $X_k$ and the $\xi_{n,i}$. Then $\zeta (t):= \tilde{X}(t)+\tilde{Y}(t)$ has the same distribution as $X(t)+Y(t)\overset{d}{=} \sum_{1 \le k \le t} f(S_k \alpha)$, and
\[ \zeta (t)=\tilde{X}(t)+ O(t^{1/2-\varepsilon}) = \sigma W(t) + O(t^{1/2-\varepsilon''}) \quad \textrm{a.s.} \]
with $\varepsilon''=\min \{ \varepsilon, \varepsilon' \} >0$, as claimed.
\end{proof}

\section{Proof of Theorem \ref{discreteKolmogorovtheorem}}\label{ratesection}

As before, $X_1, X_2, \dots$ is a sequence of i.i.d.\ nondegenerate integer-valued random variables with maximal span $D=\mathrm{gcd} (\mathrm{supp} \, X_1 - \mathrm{supp} \, X_1)$ and characteristic function $\varphi$, and $S_k = \sum_{j=1}^k X_j$. Further, $p/q$ is a reduced fraction, and $\mathbb{Z}_q = \{ 0,1/q, \dots, (q-1)/q \}$ is the cyclic group of order $q$. Let $\psi_{\mathrm{disc}}$ be as in \eqref{psidisc}.

First of all note that $\mathrm{gcd} (D,q) =1$ is equivalent to the distribution of $X_1 p/q \pmod{\mathbb{Z}}$ not being supported on a translate of a proper subgroup of $\mathbb{Z}_q$. In particular, $S_k p/q \pmod{\mathbb{Z}}$ converges in distribution to the uniform distribution on $\mathbb{Z}_q$ if and only if $\mathrm{gcd} (D,q)=1$; this is further equivalent to $\psi_{\mathrm{disc}}(k) \to 0$ as $k \to \infty$.

Consider now
\[ \psi_{\mathrm{disc}}^* (k) := \max_J |\Pr (\{S_k p/q \} \in J) -|J|/q| \]
where the maximum is over all cyclic intervals $J \subseteq \mathbb{Z}_q$, and note that $\psi_{\mathrm{disc}} (k) \le \psi_{\mathrm{disc}}^* (k) \le 2 \psi_{\mathrm{disc}} (k)$. The rate of convergence in $\psi_{\mathrm{disc}}$ and $\psi_{\mathrm{disc}}^*$ are thus the same; however, using all cyclic intervals is in a sense more natural. Indeed, since the family of all cyclic intervals is translation invariant, adding an arbitrary integer to $X_1$ does not change the value of $\psi_{\mathrm{disc}}^* (k)$; it is also not difficult to see that $\psi_{\mathrm{disc}}^*(k)$ is nonincreasing in $k$.

Most importantly, to prove Theorem \ref{discreteKolmogorovtheorem} we may assume that $D=1$. Indeed, we can first translate $X_1$ so that $\Pr (X_1 =0) >0$; this changes neither $|\varphi|$ nor the order of magnitude of $\psi_{\mathrm{disc}}(k)$. But then $D \mid X_1$ a.s., and so we can replace $X_1, X_2, \dots$ by $X_1/D, X_2/D, \dots$, and $p/q$ by $Dp/q$. Note that the characteristic function of $X_1/D$ is $\varphi (x/D)$, and that $\min_{0<|h| \le q/2} |h| \cdot \| hDp/q \| \ge A/D>0$ follows from the assumption $\mathrm{gcd} (D,q)=1$. This reduces the general case of Theorem \ref{discreteKolmogorovtheorem} to the special case $D=1$.

\subsection{Upper bounds}

The second upper bound in Theorem \ref{discreteKolmogorovtheorem} will follow from the following general upper estimate.
\begin{lem}\label{upperboundlemma} Assume that $D=1$, and that $|\varphi (2 \pi x)| \le g(x) \le e^{-cx^2}$ for all $x \in [0,r]$ with some constants $c>0$ and $0<r \le 1/2$, and some function $g$ satisfying $g(x+y)$ $\le g(x) g(y)$ whenever $x,y,x+y \in [0,r]$. Assume further, that $\min_{0<|h| \le q/2} |h| \cdot \| hp/q \| \ge A>0$. Then for all $k \ge q^2$,
\[ \psi_{\mathrm{disc}} (k) \ll \frac{g(1/q)^k}{q} I_{\{ 1/q \le r \}} + e^{-\tau k} \log q \]
with some constant $\tau >0$ and an implied constant depending only on the distribution of $X_1$, $c,r,g$ and $A$. If $r=1/2$, then the term $e^{-\tau k} \log q$ can be removed.
\end{lem}
\noindent Before we give the proof, let us make a few observations about the possible choices of the function $g$; one could call it a submultiplicative upper envelope of $|\varphi (2 \pi x)|$. Note first, that since we necessarily have $g(0)=1$, log-concavity on some interval $[0,r]$ implies submultiplicativity on the same interval.

Recalling the estimate $1-|\varphi (2 \pi x)| \gg \| Dx \|^2$ from \eqref{phiatnp2} --- which holds under the sole assumption of nondegeneracy without any moment condition ---, in the case $D=1$ we have $|\varphi (2 \pi x)| \le e^{-cx^2}$ for all $x \in [0,1/2]$ with some constant $c>0$. Because of its log-concavity, we can thus choose $g(x)=e^{-cx^2}$; in particular, Lemma \ref{upperboundlemma} applies to an arbitrary nondegenerate $X_1$.

Assuming $\E X_1^2 < \infty$, we can even choose $g(x)=|\varphi (2 \pi x)|$ with suitable constants $c>0$ and $0<r \le 1/2$. Indeed, in this case $X_1-X_2$ is a mean zero random variable with finite variance and characteristic function $\phi:=|\varphi|^2$. In particular, $\phi$ is twice continuously differentiable, $\phi'(0)=0$ and $\phi''(0)=-\mathrm{Var} (X_1-X_2)<0$. Therefore $(\log \phi)''(0) = (\phi''(0)\phi(0) - \phi'(0)^2)/\phi(0)^2 <0$; by continuity, $\phi$ is log-concave, and hence submultiplicative, in an open neighborhood of $0$.

\begin{proof}[Proof of Lemma \ref{upperboundlemma}] Let $0 \le a <q$ be an integer, and let $w: \mathbb{Z}_q \to \{ 0,1 \}$ be the indicator function of the set $\{ 0, 1/q, \dots, a/q \}$. Its $\mathbb{Z}_q$-Fourier coefficients satisfy
\[ |\hat{w}(h)| = \left| \frac{1}{q} \sum_{b=0}^a e(-hb/q) \right| \le \frac{2}{q |1-e(-h/q)|} \le \frac{1}{2|h|} \]
for any integer $0<|h| \le q/2$. Using the Fourier series expansion of $w$ we thus get
\[ \begin{split} \left| \Pr (\{ S_k p/q \} \le a/q) - (a+1)/q \right| &= \left| \E w(\{ S_k p/q\}) - (a+1)/q \right| \\ &= \bigg| \E \sum_{\substack{-q/2 < h \le q/2 \\ h \neq 0}} \hat{w}(h) e(S_k hp/q) \bigg| \\ &= \bigg| \sum_{\substack{-q/2 < h \le q/2 \\ h \neq 0}} \hat{w}(h) \varphi (2 \pi hp/q)^k \bigg| \\ &\le \sum_{\substack{-q/2 < h \le q/2 \\ h \neq 0}} \frac{|\varphi (2 \pi h p/q)|^k}{2|h|} , \end{split} \]
and consequently we have the Berry--Esseen type inequality
\begin{equation}\label{BerryEsseen}
\psi_{\mathrm{disc}} (k) \le \sum_{1 \le h \le q/2} \frac{|\varphi (2 \pi h p/q)|^k}{h} .
\end{equation}

For the rest of the proof constants and implied constants will depend only on the distribution of $X_1$, $c,r,g$ and $A$. By the assumption $D=1$, the function $|\varphi (2 \pi x)|$ is even and has smallest period $1$; in addition, $|\varphi (2 \pi x)| =1$ if and only if $x \in \mathbb{Z}$. Therefore $|\varphi (2 \pi x)| \le e^{-\tau}$ whenever $\| x \| \ge r/2$ with some constant $\tau >0$, and hence
\begin{equation}\label{hp/q>r}
\sum_{\substack{1 \le h \le q-1 \\ \| hp/q \| \ge r/2}} \frac{|\varphi (2 \pi h p/q)|^k}{h} \ll e^{-\tau k} \log q .
\end{equation}

Consider now the continued fraction representation $p/q = [a_0;a_1, \dots, a_M]$, and let $p_m/q_m = [a_0;a_1, \dots , a_m]$ denote the convergents. Since $p/q$ is reduced, we have $p_M=p$ and $q_M=q$. Let us now estimate
\[ \sum_{\substack{q_m \le h<q_{m+1} \\ \| hp/q \| < r/2}} \frac{|\varphi (2 \pi h p/q)|^k}{h} \le \frac{1}{q_m} \sum_{\substack{q_m \le h<q_{m+1} \\ \| hp/q \| < r/2}} g(\| h p/q \| )^k . \]
Let $J_0=(-\| q_m p/q \|, \| q_m p/q \|)$, $J_{\ell}=[\ell \| q_m p/q \|, (\ell+1) \| q_m p/q \| )$, $\ell =1,2,\dots$ and $J_{\ell}=((\ell-1) \| q_m p/q \| , \ell \| q_m p/q \| ]$, $\ell =-1,-2, \dots$. By the best rational approximation property, for any $0<h<q_{m+1}$ we have $\| hp/q \| \ge \| q_m p/q \|$. Consequently each interval $J_{\ell}$ contains at most one of the points $hp/q \pmod{\mathbb{Z}}$, $q_m \le h < q_{m+1}$, and the interval $J_0$ is empty. If we only consider those values of $h$ for which $\| hp/q \|<r/2$, then $J_{\ell}$ is also empty for all $|\ell | \ge r/ (2\| q_m p/q \|)$. Note that the submultiplicative assumption on $g$ in particular implies that $g(x)$ is nonincreasing on $[0,r]$, and that all nonempty $J_{\ell}$ is a subset of $[-r,r]$; therefore,
\[ \frac{1}{q_m} \sum_{\substack{q_m \le h<q_{m+1} \\ \| hp/q \| < r/2}} g( \| h p/q \| )^k \le \frac{2}{q_m} \sum_{1 \le \ell < r/(2\| q_m p/q \| )} g(\ell \| q_m p/q \| )^k . \]
By the assumptions on $g$, consecutive terms in the previous sum satisfy
\[ \frac{g((\ell +1) \| q_m p/q \| )^k}{g(\ell \| q_m p/q \| )^k} \le g(\| q_m p/q \| )^k \le e^{-ck \| q_m p/q \|^2} \le e^{-ck/q^2} \le e^{-c} . \]
The terms thus decay exponentially fast, hence
\[ \frac{2}{q_m} \sum_{1 \le \ell < r/(2 \| q_m p/q \| )} g(\ell \| q_m p/q \| )^k \ll \frac{g(\| q_m p/q \| )^k}{q_m} , \]
and by summing over $m$ we get
\[ \sum_{\substack{1 \le h \le q-1 \\ \| hp/q \| < r/2}} \frac{|\varphi (2 \pi h p/q)|^k}{h} \ll \sum_{\substack{m=1 \\ \| q_m p/q \| <r/2}}^{M-1} \frac{g(\| q_m p/q \| )^k}{q_m} . \]
Here the last term dominates. Indeed, by the assumptions on $g$,
\[ \begin{split} \sum_{\substack{m=1 \\ \| q_m p/q \| <r/2}}^{M-1} \frac{g(\| q_m p/q \| )^k}{q_m} &= \frac{g(\| q_{M-1} p/q \| )^k}{q_{M-1}} \sum_{\substack{m=1 \\ \| q_m p/q \| <r/2}}^{M-1} \frac{q_{M-1}}{q_m} \cdot \frac{g(\| q_m p/q \| )^k}{g(\| q_{M-1} p/q \| )^k} \\ &\le \frac{g(\| q_{M-1} p/q \| )^k}{q_{M-1}} \sum_{\substack{m=1 \\ \| q_m p/q \| <r/2}}^{M-1} \frac{q_{M-1}}{q_m} \cdot e^{-c(\| q_m p/q \| -\| q_{M-1} p/q \|)^2 k} \\ &\le \frac{g(\| q_{M-1} p/q \| )^k}{q_{M-1}} \sum_{\substack{m=1 \\ \| q_m p/q \| <r/2}}^{M-1} \frac{q_{M-1}}{q_m} \cdot e^{-c(q \| q_m p/q \| -1)^2} . \end{split} \]
Using $\| q_m p/q \| \ge 1/(q_{m+1}+q_m)$ it is not difficult to see that the terms of the last sum decay exponentially fast as $m$ decreases, showing that the sum is $\ll 1$. In the main factor we have $p/q=p_M/q_M$. Recalling the identity $p_m q_{m-1} -p_{m-1}q_m = (-1)^{m-1}$ from the theory of continued fractions, we get that $\| q_{M-1}p_M/q_M \| = 1/q_M$, hence we altogether deduce
\[ \sum_{\substack{1 \le h \le q-1 \\ \| hp/q \| < r/2}} \frac{|\varphi (2 \pi h p/q)|^k}{h} \ll \frac{g(\| q_{M-1} p/q \| )^k}{q_{M-1}} I_{\{ \| q_{M-1} p/q \| <r/2 \}} \ll \frac{g(1/q)^k}{q} I_{\{ 1/q \le r \}} . \]
The last relation together with \eqref{hp/q>r} shows
\[ \sum_{h=1}^{q-1} \frac{|\varphi (2 \pi h p/q)|^k}{h} \ll \frac{g(1/q)^k}{q} I_{\{ 1/q \le r \}} + e^{-\tau k} \log q , \]
and the claim follows from the Berry--Esseen type inequality \eqref{BerryEsseen}.

In the case $r=1/2$ (when $g$ is submultiplicative on the whole interval $[0,1/2]$) we can repeat the same proof without separating the terms $\| hp/q \| \ge r/2$ and $\| hp/q \| < r/2$, and obtain $\psi_{\mathrm{disc}}(k) \ll g(1/q)^k/q$.
\end{proof}

\subsection{Lower bounds}

The lower bounds in Theorem \ref{discreteKolmogorovtheorem} will follow from the following lemma.
\begin{lem}\label{lowerboundlemma} We have
\[ \psi_{\mathrm{disc}} (k) \ge \frac{|\varphi (2 \pi /q)|^k}{2(q-1)} \qquad \textrm{for all } k \ge 1 . \]
Assuming that $\E X_1^2 < \infty$, we also have
\[ \psi_{\mathrm{disc}} (k) \ge \frac{1}{12\sqrt{3 (\mathrm{Var} \, X_1) k}+6} \qquad \textrm{for all } 1 \le k \le \frac{(q-3)^2}{108 \mathrm{Var} \, X_1} . \]
\end{lem}

\begin{proof} To prove the first claim, we use a discrete version of Koksma's inequality: for any $\mathbb{Z}_q$-valued random variables $X$ and $Y$, and any $f: \mathbb{Z}_q \to \mathbb{C}$,
\begin{equation}\label{discreteKoksma}
\left| \E f(X) - \E f(Y) \right| \le V_q(f) \max_{0 \le a <q} \left| \Pr (X  \le a/q) - \Pr (Y \le a/q) \right|
\end{equation}
where $V_q(f) = \sum_{a=0}^{q-2} |f((a+1)/q) - f(a/q) |$ is the $\mathbb{Z}_q$-total variation of $f$. The inequality \eqref{discreteKoksma} follows from a simple summation by parts. Now let $0<h<q$ be an integer, and note that $f(x/q)=e(h x/q)$ has $\mathbb{Z}_q$-total variation
\[ V_q (f) = \sum_{a=0}^{q-2} |e(h (a+1)/q) - e(h a/q)| = (q-1) |2 \sin (\pi h/q)| \le 2(q-1) . \]
Applying \eqref{discreteKoksma} with $X=\{ S_k p/q \}$ and a uniformly distributed $Y$, and noting that $\E e(h Y) =0$, we get
\[ \left| \E e(S_k hp/q) \right| \le 2(q-1) \psi_{\mathrm{disc}}(k) . \]
Choosing $h$ to be the multiplicative inverse of $p$ modulo $q$, we finally obtain $|\varphi (2 \pi /q)|^k = |\E e( S_k /q)| \le 2(q-1) \psi_{\mathrm{disc}}(k)$, as claimed.

Next, assume that $\E X_1^2 < \infty$, and let $1 \le k \le (q-3)^2 / (108 \mathrm{Var}\, X_1)$. By the Chebyshev inequality, $S_k$ lies in an interval of length $2 \sqrt{3 (\mathrm{Var} \, X_1)k}$ with probability at least $2/3$. There are at most $2 \sqrt{3 (\mathrm{Var} \, X_1)k}+1$ integers in this interval, hence by the pigeonhole principle $\Pr (S_k =n) \ge (3\sqrt{3 (\mathrm{Var} \, X_1)k}+3/2)^{-1}$ for some $n$. It follows that the distribution of $\{ S_k p/q \}$ has an atom, say $a/q \in \mathbb{Z}_q$, of weight $\ge (3\sqrt{3 (\mathrm{Var} \, X_1)k}+3/2)^{-1}$. Letting $F(a/q) = \Pr (\{ S_k p/q \} \le a/q)$ we thus have
\[ \begin{split} \frac{1}{3\sqrt{3 (\mathrm{Var} \, X_1)k}+3/2} &= \left| F(a/q) - F((a-1)/q) \right| \\ &\le \left| F(a/q) -(a+1)/q \right| + |F((a-1)/q) -a/q| + 1/q . \end{split} \]
Here $1/q$ is at most half of the left hand side, and consequently either $| F(a/q) -(a+1)/q|$ or $|F((a-1)/q) -a/q|$ is at least $(12\sqrt{3(\mathrm{Var} \, X_1)k}+6)^{-1}$. In particular, $\psi_{\mathrm{disc}} (k) \ge (12\sqrt{3 (\mathrm{Var} \, X_1) k}+6)^{-1}$, as claimed. 
\end{proof}

\begin{proof}[Proof of Theorem \ref{discreteKolmogorovtheorem}] As observed at the beginning of Section \ref{ratesection}, we may assume that $D=1$. First, we prove the upper bounds, starting with the case $k \le q^2$. In \cite{BB1} we showed that if $\alpha$ is a badly approximable irrational and $\E X_1^2 < \infty$, then $\sup_{0 \le x \le 1} |\Pr (\{ S_k \alpha \} \le x)-x| \ll k^{-1/2}$ with implied constant depending only on the maximal partial quotient in the continued fraction of $\alpha$. In fact, the same proof works for a rational $p/q$ in place of $\alpha$, provided that $k \le q^2$; in particular, $\sup_{0 \le x \le 1} |\Pr (\{ S_k p/q \} \le x)-x| \ll k^{-1/2}$. Since the uniform distribution on $\mathbb{Z}_q$ also has distance $\ll q^{-1} \ll k^{-1/2}$ from the Lebesgue measure in the (continuous) Kolmogorov metric, it follows that $\psi_{\mathrm{disc}}(k) \ll k^{-1/2}$.

Next, let $k>q^2$. As observed, we can apply Lemma \ref{upperboundlemma} with $g(x)=|\varphi (2 \pi x)|$ and suitable constants $c>0$ and $0<r \le 1/2$ to obtain
\[ \psi_{\mathrm{disc}} (k) \ll \frac{|\varphi (2 \pi /q)|^k}{q} I_{\{ 1/q \le r \}} + e^{-\tau k} \log q . \]
It is easy to see that there exists a constant $q_0$ depending only on the distribution of $X_1$ such that for all $q \ge q_0$ and all $k>q^2$, the second term is negligible compared to the first one; in particular, $\psi_{\mathrm{disc}}(k) \ll |\varphi (2 \pi /q)|^k/q$, as claimed.

Finally, we prove the lower bounds. Lemma \ref{lowerboundlemma} immediately shows that the claim holds for any $k \le Cq^2$ with some constant $C>0$, and also for all $k>q^2$. To see the claim on the remaining interval $Cq^2 \le k \le q^2$, simply note that the asymptotics $|\varphi (2 \pi x)|^2=1-4 \pi^2 (\mathrm{Var} \, X_1) x^2 (1+o(1))$ as $x \to 0$ shows that in this case
\[ \psi_{\mathrm{disc}}(k) \ge \frac{|\varphi (2 \pi /q)|^k}{2(q-1)} \ge \frac{(1-8 \pi^2 (\mathrm{Var} \, X_1)/q^2)^{k/2}}{2(q-1)} \gg k^{-1/2}, \]
provided that $q$ is large enough.
\end{proof}

\subsection{Total variation metric}

By similar arguments as those at the beginning of Section \ref{ratesection}, we have $\psi_{\mathrm{TV}}(k) \to 0$ if and only if $\mathrm{gcd}(D,q)=1$. In addition, to prove \eqref{psiTV}, we may assume that $D=1$. By the so-called ``upper bound lemma'' in \cite{CDG},
\[ \psi_{\mathrm{TV}} (k)^2 \le \frac{1}{4} \sum_{h=1}^{q-1} |\E e(S_k h p/q)|^2 = \frac{1}{4} \sum_{h=1}^{q-1} |\varphi (2 \pi h/q)|^{2k} . \]
In the second step we used the fact that multiplication by $p$ is a bijection of all nonzero remainders modulo $q$. As observed in the paragraph after Lemma \ref{upperboundlemma}, the function $|\varphi (2 \pi x)|$ is submultiplicative on $[0,r]$, and $|\varphi (2 \pi x)| \le e^{-\tau}$ whenever $\| x \| \ge r$ with suitable constants $0<r \le 1/2$ and $\tau >0$. Following the methods in the proof of Lemma \ref{upperboundlemma}, we get
\[ \sum_{\substack{-rq \le h \le rq \\ h \neq 0}} |\varphi (2 \pi h/q)|^{2k} \le 2 \sum_{1 \le h \le rq} |\varphi (2 \pi /q)|^{2hk} \le 2 \frac{|\varphi (2 \pi /q)|^{2k}}{1-|\varphi (2 \pi /q)|^{2k}} , \]
and
\[ \sum_{rq<|h| \le q/2} |\varphi (2 \pi h/q)|^{2k} \le e^{-2 \tau k} q ; \]
therefore,
\begin{equation}\label{psiTVupperbound}
\psi_{\mathrm{TV}} (k)^2 \le \frac{1}{2} \cdot \frac{|\varphi (2 \pi /q)|^{2k}}{1-|\varphi (2 \pi /q)|^{2k}} + \frac{1}{4} e^{-2 \tau k} q .
\end{equation}
It is now easy to see that for all large enough $q$ and $k \ge 1$ we have $\psi_{\mathrm{TV}}(k) \ll |\varphi (2 \pi /q)|^k$, as claimed. Indeed, if, say, $|\varphi (2 \pi /q)|^{2k} \ge 1/2$, then the claim follows from the trivial estimate $\psi_{\mathrm{TV}}(k) \le 1$. If $|\varphi (2 \pi /q)|^{2k} < 1/2$, then we necessarily have $k \gg q^2$, and consequently $e^{-2 \tau k} q \ll |\varphi (2 \pi /q)|^{2k}$ for all large enough $q$; the claim then follows from \eqref{psiTVupperbound}. This finishes the proof of the upper bound in \eqref{psiTV}.

To see the lower bound in \eqref{psiTV}, simply note that $f(x/q)=e(hx/q)$, where $h$ is the multiplicative inverse of $p$ modulo $q$, has maximum norm $1$ on $\mathbb{Z}_q$. Hence
\[ \psi_{\mathrm{TV}} (k) \ge \frac{1}{2} |\E e(S_k hp/q)| = \frac{1}{2} |\varphi (2 \pi /q)|^k , \]
as claimed.

\subsection*{Acknowledgments}

Istv\'an Berkes is supported by NKFIH grant K 125569. Bence Borda is supported by the Austrian Science Fund (FWF), project Y-901.

\end{document}